\documentclass[11pt,reqno]{amsart}
\usepackage{amsmath, amssymb, amsthm}
\usepackage{url}
\usepackage[breaklinks]{hyperref}

\newcommand{\Res}{\operatorname{Res}}
\renewcommand{\Re}{\operatorname{Re}}

 \renewcommand{\a}{\alpha}
\renewcommand{\b}{\beta}

\renewcommand{\r}{{\rho}}

\renewcommand{\(}{\left\(}
\renewcommand{\)}{\right\)}
\renewcommand{\[}{\left\[}
\renewcommand{\]}{\right\]}

\numberwithin{equation}{section}
 \theoremstyle{plain}
\newtheorem{theorem}{Theorem}[section]
\newtheorem{lemma}[theorem]{Lemma}
\newtheorem{remark}[]{Remark}

\newtheorem{corollary}[theorem]{Corollary}

   \makeatletter
\def\proof{\@ifnextchar[{\@oproof}{\@nproof}}
\def\@oproof[#1][#2]{\trivlist\item[\hskip\labelsep\textit{#2 Proof of\
#1.}~]\ignorespaces}
\def\@nproof{\trivlist\item[\hskip\labelsep\textit{Proof.}~]\ignorespaces}

\makeatother

\begin{document}
\title[On Ramanujan's formula for $\zeta(1/2)$ and $\zeta(2m+1)$ ]{On Ramanujan's formula for $\zeta(1/2)$ and $\zeta(2m+1)$}

\author{Anushree Gupta}
\address{Anushree Gupta\\ Department of Mathematics \\
Indian Institute of Technology Indore \\
Simrol,  Indore,  Madhya Pradesh 453552, India.} 
\email{anushreegupta251@gmail.com,  msc1903141004@iiti.ac.in}

 \author{Bibekananda Maji}
\address{Bibekananda Maji\\ Department of Mathematics \\
Indian Institute of Technology Indore \\
Simrol,  Indore,  Madhya Pradesh 453552, India.} 
\email{bibek10iitb@gmail.com,  bibekanandamaji@iiti.ac.in}


\maketitle

\begin{abstract}
Page 332 of Ramanujan's Lost Notebook contains a compelling identity for $\zeta(1/2)$,  which has been studied by many mathematicians over the years. On the same page, Ramanujan also recorded the series, 
\begin{align*}
\frac{1^r}{\exp(1^s x) - 1} + \frac{2^r}{\exp(2^s x) - 1} + \frac{3^r}{\exp(3^s x) - 1} + \cdots,
\end{align*}
where $s$ is a positive integer and $r-s$ is {\it any even} integer.  Unfortunately,  Ramanujan doesn't give any formula for it. This series was rediscovered by Kanemitsu, Tanigawa, and Yoshimoto,  although they studied it only when $r-s$ is a {\it negative even} integer.  Recently, Dixit and the second author generalized the work of Kanemitsu et al. and obtained a transformation formula for the aforementioned series with $r-s$ is {\it any even} integer. While extending the work of Kanemitsu et al.,  Dixit and the second author obtained a beautiful generalization of Ramanujan's formula for odd zeta values.  In the current paper,  we investigate transformation formulas for an infinite series, and interestingly, we derive Ramanujan's formula for $\zeta(1/2)$, Wigert's formula for $\zeta(1/k)$ as well as Ramanujan's formula for $\zeta(2m+1)$. Furthermore, we obtain a new identity for $\zeta(-1/2)$ in the spirit of Ramanujan.   
\end{abstract}

\section{Introduction}

The study of understanding the algebraic nature of the special values of the Riemann zeta function $\zeta(s)$ has been a constant source of motivation for mathematicians to produce many beautiful results.   In 1735,  Euler established the following elegant formula for $\zeta(2m)$.
 For every positive integer $m$, 
\begin{equation}\label{Euler_zeta(2m)}
\zeta(2 m ) = (-1)^{m +1} \frac{(2\pi)^{2 m}B_{2 m }}{2 (2 m)!},
\end{equation}
where $B_{2m}$ denotes the $2m$th Bernoulli number.  
These  are certain rational numbers,  
defined by the generating function:
$${\frac{z}{e^z - 1} = \sum_{n=0}^\infty B_{n} \frac{z^n}{n!}, \quad |z| < 2\pi.}  
$$
Euler's formula along with the transcendence of $\pi$, proved by F. Lindemann,  immediately tells us that all even zeta values are transcendental. However, the arithmetic nature of the odd zeta values is yet to be determined.  Surprisingly in 1979,  Roger Ap\'{e}ry \cite{Apery79, apery2}  proved the irrationality of $\zeta(3)$.  But till today we do not know about the algebraic nature of $\zeta(3)$,  that is,  whether $\zeta(3)$ is algebraic or transcendental.  
In 2001,  Rivoal \cite{rivoal} Ball and Rivoal  \cite{ballrivoal} made a breakthrough by proving that there exist infinitely many odd zeta values which are irrational,  but their result does not tells us anything about the algebraic nature of a specific odd zeta value.
Around the same time,  Zudilin \cite{zudilin} proved that at least one of the members of $\zeta(5), \zeta(7), \zeta(9)$ or $\zeta(11)$ is irrational,  which is the current best result in this direction.  

Before going to England,  Ramanujan, in his second notebook  \cite[p.~173, Ch. 14, Entry 21(i)]{ramnote}, noted down the following remarkable formula for odd zeta values $\zeta(2 m + 1)$.

Let $\alpha$ and $\beta$ be two positive numbers such that $\alpha \beta = \pi^2$.  For any non-zero integer $m$,  
\begin{align}\label{Ramanujan's formula}
\frac{1}{(4\alpha)^m}  \left( \frac{1}{2} \zeta(2 m +1)  +  \sum_{n=1}^{\infty} \frac{n^{-2m-1}}{(e^{2n\alpha} -1 ) } \right) 
& +  \frac{(-1)^{m+1}}{(4 \beta)^m}  \left( \frac{1}{2} \zeta(2m+1) + \sum_{n=1}^{\infty} \frac{n^{-2m-1}}{(e^{2n\beta } -1 ) }   \right) \nonumber \\
&  = \sum_{k=0}^{m+1}  (-1)^{k-1} \frac{B_{2k}}{(2k)!} \frac{B_{2m+2-2k}}{(2m+2-2k)!} \alpha^{m+1-k} \beta^k.
\end{align}
This formula can also be found in Ramanujan's Lost Notebook \cite[pp.~319-320, formula (28)]{lnb}. 
In sharp contrast to Euler's formula, it falls short of providing an explicit formula for odd zeta values, 
but it is indeed another wonderful discovery of Ramanujan that has attracted the attention of several mathematicians.
The first published proof was given by Malurkar \cite{malurkar}  in 1925,  who had no idea about its presence in Ramanujan's Notebook.
Lerch, in 1901,  proved a particular case of the formula \eqref{Ramanujan's formula},  namely,   corresponding to $\alpha=\beta= \pi$ and $m$ an odd positive integer.  Replacing $m$ by $2m+1$ in \eqref{Ramanujan's formula},  it reduces to
\begin{align}\label{Lerch}
\zeta(4m+3)+2\sum_{n=1}^{\infty}\frac{n^{-4m-3}}{(e^{2\pi n}-1)}=\pi^{4m+3}2^{4m+2}\sum_{j=0}^{2m+2}\frac{(-1)^{j+1}B_{2j}B_{4m+4-2j}}{(2j)!(4m+2-2j)!}. 
\end{align}
Note that the infinite series present in the left side of \eqref{Lerch} is a rapidly convergent series,  which indicates that $\zeta(4m+3)$ is almost a rational multiple of $\pi^{4m+3}$.  Moreover,  one can say that atleast one of the expressions $\zeta(4m+3)$ or $\sum_{n=1}^{\infty}\frac{n^{-4m-3}}{(e^{2\pi n}-1)}$ must  be transcendental since the right side expression is transcendental.  A folklore conjecture is that $\pi$ and odd zeta values do not satisfy any non-zero polynomial with rational coefficients,  which implies that all odd zeta values must be transcendental.  

Ramanujan's formula \eqref{Ramanujan's formula} has a deep connection with the theory of  Eisenstien series over $SL_{2}(\mathbb{Z})$ and their Eichler integrals.  To know more about this connection,  one can see \cite{berndt-straub},  \cite{GMR}.  In 1977,  Bruce Berndt  \cite{berndt77} obtained a  modular transformation formula for a generalized Eisenstien series from which he was able to derive Euler's formula for $\zeta(2m)$ as well as Ramanujan's formula for $\zeta(2m+1)$.  Over the years, many mathematicians found generalizations of \eqref{Ramanujan's formula} in various directions.  Readers are encouraged to see the paper of Berndt and Straub \cite{berndtstraubzeta} for more information on the history of Ramanujan's formula for odd zeta values.  

In 2001,  Kanemitsu,  Tanigawa and Yoshimoto  \cite{KTY01} studied a generalized Lambert series defined by
\begin{equation}\label{Kanemitsu et al}
\sum_{n=1}^{\infty} \frac{n^{N-2h}}{\exp(n^N x) - 1},
\end{equation}
where $h,  N \in \mathbb{N}$ such that  $1 \leq h \leq N/2$.  They established an important modular transformational relation for this generalized Lambert series which enabled them to find a formula for the Riemann zeta function at rational arguments.  Recently,  Dixit and the second author \cite{DM20} pointed out that the above generalized Lambert series is in fact present on page 332 of Ramanujan's Lost Notebook \cite{lnb} with more general conditions on the parameters.  At the end of page 332,  Ramanujan wrote 
\begin{align}\label{Ramanujan_original}
\frac{1^r}{\exp(1^s x) - 1} + \frac{2^r}{\exp(2^s x) - 1} + \frac{3^r}{\exp(3^s x) - 1} + \cdots,
\end{align}
where $s$ is a positive integer and $r-s$ is {\it any even} integer.  Unfortunately,  Ramanujan does not give any formula for the series \eqref{Ramanujan_original}.  Comparing \eqref{Kanemitsu et al} and \eqref{Ramanujan_original},  one can observe that Kanemitsu et al.  studied \eqref{Ramanujan_original} with the restriction that  $r-s$ is a {\it negative even} integer that lies in a restricted domain.   This motivated Dixit and the second author to generalize the main result of Kanemitsu et al. \cite[Theorem 1.1]{KTY01} and find a formula for \eqref{Ramanujan_original}.  While extending the main result of Kanemitsu et al.,  Dixit and the second author \cite[Theorem 1.2]{DM20} obtained the following beautiful generalization of the Ramanujan's formula for $\zeta(2m+1)$.  

Let $N \geq 1 $ be an odd positive integer and $\a,\b$ be two positive real numbers such that $\a\b^{N}=\pi^{N+1}$.  Then for any integer non-zero integer $m \neq 0$,  we have
\begin{align}\label{zetageneqn}
&\a^{-\frac{2Nm}{N+1}}\left(\frac{1}{2}\zeta(2Nm+1)+\sum_{n=1}^{\infty}\frac{n^{-2Nm-1}}{\textup{exp}\left((2n)^{N}\a\right)-1}\right)\nonumber\\
&=\left(-\b^{\frac{2N}{N+1}}\right)^{-m}\frac{2^{2m(N-1)}}{N}\Bigg(\frac{1}{2}\zeta(2m+1)+(-1)^{\frac{N+3}{2}}\sum_{j=\frac{-(N-1)}{2}}^{\frac{N-1}{2}}(-1)^{j}\sum_{n=1}^{\infty}\frac{n^{-2m-1}}{\textup{exp}\left((2n)^{\frac{1}{N}}\b e^{\frac{i\pi j}{N}}\right)-1}\Bigg)\nonumber\\
&\quad+(-1)^{m+\frac{N+3}{2}}2^{2Nm}\sum_{j=0}^{\left\lfloor\frac{N+1}{2N}+m\right\rfloor}\frac{(-1)^jB_{2j}B_{N+1+2N(m-j)}}{(2j)!(N+1+2N(m-j))!}\a^{\frac{2j}{N+1}}\b^{N+\frac{2N^2(m-j)}{N+1}}.
\end{align}
This formula is of interest as it establishes a relation between two distinct odd zeta values,  namely $\zeta(2m+1)$ and $\zeta(2Nm+1)$ when $N>1$.
To derive the above generalization of Ramanujan's formula for $\zeta(2m+1)$,  Dixit and the second author studied  the following line integration representation of \eqref{Kanemitsu et al}:
\begin{align*}
\sum_{n=1}^{\infty} \frac{n^{N-2h}}{\exp(n^N x) - 1} =  \frac{1}{2\pi i} \int_{c_0-i \infty}^{c_0+ i \infty} \Gamma(s) \zeta(s) \zeta(Ns-(N-2h)) x^{-s} \mathrm{d}s,
\end{align*}
where $c_0 > \max(1, (N-2h+1)/N)$.  

Inspired by the work of Dixit and the second author,  in the current paper,  we would like to investigate the following line integration: 
\begin{align*}
 \frac{1}{2\pi i} \int_{c-i \infty}^{c+ i \infty} \Gamma(s) \zeta(ks) \zeta(s-r) x^{-s} \mathrm{d}s,
\end{align*}
where $k\in \mathbb{N},  r \in \mathbb{Z}$ and $c > \max\left(  1/k, 1+r \right)$.  
Surprisingly,  when $k-r$ is an  {\it even} integer,  we obtain nice transformation formulas which allow us to derive many well-known formulas in the literature,  for example,  Ramanujan's formula for $\zeta(1/2)$,  Wigert's formula for $\zeta(1/k)$ for $k\geq$ 2 even,  and Ramanujan's formula for odd zeta values. 
Furthermore,  we also obtain a new identity for $\zeta(-1/2)$ analogous to Ramanujan's formula for $\zeta(1/2)$.

\section{Preliminaries}

The study of the well-known divisor function $d(n)$ plays an important  role in analytic number theory.  This function has been generalized in many directions, and one of the important generalizations is defined by
\begin{align*}
\sigma_{k}(m) := \sum_{d|m} d^k,
\end{align*}
where $k\in \mathbb{Z}$.  
In this paper,  we encountered two new divisor functions that are defined by
\begin{align}\label{definition of D_{k,r}}
D_{k,r}(n) :=\sum_{d^k|n} \left(\frac{n}{d^k}\right)^r, \quad {\rm and} \quad S_{k,r}(n) :=\sum_{d^k|n} \left(\frac{n}{d^k} \right)^{-r}d^{k-1},
\end{align}
where $k \in \mathbb{N}$ and $r \in \mathbb{Z}$.  At this point,  it may look artificial to define the above two divisor functions, but in due time we will see that these divisor functions naturally arise with the theory.  One can easily observe that 
\begin{align}\label{relation_D_sigma}
D_{1, r}(n) = \sigma_{r}(n),  \quad {\rm and}\quad S_{1,r}= \sigma_{-r}(n). 
\end{align}
It is well-known that $  \zeta(s) \zeta(s-k) = \sum_{n=1}^{\infty} \frac{\sigma_{k}(n)}{n^s}$, absolutely convergent for $\Re(s) > {\rm max }\{1, k+1 \} $.
We now look for the Dirichlet series associated to $D_{k,r}(n)$ and $S_{k,r}(n)$ in terms of the Riemann zeta function.  Using the definition \eqref{definition of D_{k,r}},  we  see that
\begin{align}\label{gen_D(k,r)}
\sum_{n=1}^{\infty}\frac{D_{k,r}(n)}{n^s}
 =\sum_{n=1}^{\infty}\left(\sum_{n_{1}^k|n} \left(\frac{n}{n_{1}^k} \right)^r \right) \frac{1}{n^s}
= \sum_{n_{1}=1}^{\infty} \frac{1}{n_{1}^{ks}} \sum_{n_{2}=1}^{\infty} \frac{1}{n_{2}^{s-r}}
=\zeta(ks)\zeta(s-r), 
\end{align}
 valid for $ \Re(s)>\max\left(\frac{1}{k}, 1+r \right)$.  Similarly,  with the help of the definition \eqref{definition of D_{k,r}}, one can show that,  for $\Re(s) > \max(1,  1-r)$, 
\begin{align}\label{gen_S(k,r)}
\sum_{n=1}^{\infty}\frac{S_{k,r}(n)}{n^s}  = \zeta(1+ks-k)\zeta(s+r).
\end{align}
We observe that the Lambert series $\sum_{m=1}^{\infty}\sigma_{k}(m)e^{-mx}$ associated to the divisor function $\sigma_{k}(m)$ is in fact present in Ramanujan's formula \eqref{Ramanujan's formula} as it can be clearly seen from the relation that 
\begin{align*}
\sum_{m=1}^{\infty}\sigma_{k}(m)e^{-mx} =\sum_{m=1}^{\infty}\left(\sum_{d|m}d^{k}\right)e^{-mx} =\sum_{m=1}^{\infty}\sum_{d=1}^{\infty}d^k e^{-mdx} 
= \sum_{n=1}^{\infty}\frac{n^k}{e^{nx}-1}. 
\end{align*}
Let us try to find a similar expression for the Lambert series $\sum_{n=1}^{\infty}D_{k,r}(n)e^{-nx}$ associated to the new divisor function $D_{k,r}(n)$. 
We know that,  for $x>0$,  $ \sum_{m=1}^{\infty} e^{-n^{k}mx}=\frac{1}{e^{n^{k}x}-1}$. Now  differentiating both sides of this series $r$ times, and then taking sum over $n$, we obtain
\begin{align}\label{D(k,r) for r positive}
&  \sum_{m=1}^{\infty}  (-1)^{r}(n^{k} m)^{r} e^{-n^{k}mx}  = \frac{\mathrm{d}^r }{\mathrm{d}x^r}\left(\frac{1}{e^{n^{k}x}-1}\right)\nonumber\\
\Rightarrow & \sum_{n=1}^\infty \sum_{m=1}^{\infty}  m^{r} e^{-n^{k}mx}  = \sum_{n=1}^\infty \frac{(-1)^r}{n^{kr}}\frac{\mathrm{d}^r }{\mathrm{d}x^r}\left(\frac{1}{e^{n^{k}x}-1}\right)\nonumber\\
 \Rightarrow &  \sum_{n=1}^{\infty}D_{k,r}(n)e^{-nx}=\sum_{n=1}^{\infty}\frac{(-1)^r}{n^{kr}}\frac{\mathrm{d}^r }{\mathrm{d}x^r}\left(\frac{1}{e^{n^{k}x}-1}\right),\quad {\rm for}\,\,  {r\geq0}.
\end{align}

Next, we state another important special function,  the Meijer $G$-function,  which is the generalization of many well-known special functions in the literature.  

Let $0\leq m \leq q$, $0\leq n \leq p$.  Let $a_1, \cdots, a_p$ and $b_1, \cdots, b_q$ complex numbers such that $a_i - b_j \not\in \mathbb{N}$ for $1 \leq i \leq n$ and $1 \leq j \leq m$.  Then the Meijer $G$-function \cite[p.~415, Definition 16.17]{NIST}  is defined by the  line integral:
\begin{align}\label{MeijerG}
G_{p,q}^{\,m,n} \!\left(  \,\begin{matrix} a_1,\cdots , a_p \\ b_1, \cdots , b_q \end{matrix} \; \Big| z   \right) = \frac{1}{2 \pi i} \int_L \frac{\prod_{j=1}^m \Gamma(b_j + s) \prod_{j=1}^n \Gamma(1 - a_j -s) z^{-s}  } {\prod_{j=m+1}^q \Gamma(1 - b_j - s) \prod_{j=n+1}^p \Gamma(a_j + s)}\mathrm{d}s,
\end{align}
where the vertical line of integration $L$,  going from $-i \infty$ to $+i \infty$, separates the poles of the factors $\Gamma(1-a_j-s)$ from those of the factors   $\Gamma(b_j+s)$. This integral converges  if $p+q < 2(m+n)$ and $|\arg(z)| < (m+n - \frac{p+q}{2}) \pi$.

Now we are ready to state the main results.


\section{Main results}

\begin{theorem}\label{k even r even}

Let $k \geq 2$ and $r$ be  an  even  integers.  Let $D_{k,r}(n)$ and $S_{k,r}(n)$ be defined as in \eqref{definition of D_{k,r}}. 
Then for any $x>0$, we have
 \begin{align*}
 \sum_{n=1}^{\infty} D_{k,r}(n) e^{-nx}=&-\frac{1}{2}\zeta(-r)+\frac{1}{k}\Gamma\left(\frac{1}{k}\right)\zeta\left(\frac{1}{k}-r\right)x^{- \frac{1}{k}}+R_{1+r}+\frac{(-1)^{\frac{k+r-2}{2}} (2\pi)^{\frac{k+1-2r}{2}}}{x\, k^{\frac{2k-1}{2}}}
  \\& \times \sum_{j=-(k-1)}^{(k-1)} {}^{''} \sum_{n=1}^{\infty} S_{k,r}(n)\; G_{0,k}^{\,k,0} \!\left(  \,\begin{matrix}\{\}\\r,-\frac{1}{k},\cdots ,-\frac{(k-1)}{k} \end{matrix} \; \Big| X(j)   \right),
 \end{align*}
 where $''$ means summation runs over $ j=-(k-1),-(k-3), \cdots ,(k-3),(k-1)$,  and
\begin{align}\label{X(j)}
 X(j):= X_{x,n,k}(j) :=\frac{e^{-\frac{i\pi j}{2}} (2\pi)^{k+1}n}{k^{k} x},  
 \end{align} 
 and
 \begin{align}\label{R(1+r)}
  R_{1+r} =
  \begin{cases}
    r! \zeta (k(1+r )) x^{-(1+r)}, & {\rm if  }\,\, r \geq 0, \\
   \frac{(-1)^{1+r}}{(-(1+r))!} k \zeta' (k(1+r)) x^{-(1+r)}, & {\rm if}\,\,   r<0. 
  \end{cases}
\end{align}
\end{theorem}
As an immediate consequence of the above identity,  we obtain Ramanujan's formula for $\zeta(\frac{1}{2})$. 
\begin{corollary}\label{zeta(1/2)} 
Let $\alpha$ and $\beta$ be two positive numbers such that $\alpha \beta=4\pi^{3}$.Then
\begin{align}\label{Ramanujan_zeta(1/2)}
\sum_{n=1}^{\infty}\frac{1}{e^{n^{2}\alpha}-1}=\frac{1}{4}+\frac{\pi^2}{6 \alpha}+\frac{\sqrt{\beta}}{4\pi}\zeta\left(\frac{1}{2}\right)+\frac{\sqrt{\beta}}{4\pi}\sum_{m=1}^{\infty}   \frac{\cos(\sqrt{m\beta})-\sin(\sqrt{m\beta})-e^{-\sqrt{m\beta}}} { \sqrt{m} ( \cosh(\sqrt{m\beta})-\cos(\sqrt{m\beta}))}.
\end{align}
\end{corollary}
Ramanujan\footnote{Ramanujan \cite[p.~332]{lnb} missed the factor $1/\sqrt{m}$ on the right side series expression of \eqref{Ramanujan_zeta(1/2)}.} recorded this formula on the page 332 of his Lost Notebook \cite{lnb},  where Dixit and the second author encountered the series \eqref{Ramanujan_original}.  On the same page, Ramanujan mentioned another form of this identity, which can also be found in Ramanujan's second notebook \cite[Entry 8, Chapter 15]{ramnote} and in \cite[p.~314]{bcbramsecnote}.  The formula \eqref{Ramanujan_zeta(1/2)} has been generalized by many mathematicians,  for more information readers can see \cite[pp.~191-193]{andrews-berndtIV},  \cite[p.~859,  Theorem 10.1]{pathways}.

More generally,  substituting $r=0$ in Theorem \ref{k even r even}, we obtain the following Wigert's formula \cite{wigert} for $\zeta(\frac{1}{k})$, for $k\geq2$ even. 
\begin{corollary}\label{wigert formula for zeta(1/k)}
For  any $x>0$ and $k\geq 2$ even,  we have
\begin{align}
\sum_{n=1}^{\infty}\frac{1}{e^{n^{k}x}-1}=&\frac{\zeta(k)}{x}+\frac{1}{k}\Gamma\left(\frac{1}{k}\right)\zeta\left(\frac{1}{k}\right)\left(\frac{1}{x}\right)^{\frac{1}{k}}+\frac{1}{4}+  \frac{(-1)^{\frac{k}{2}-1}}{k}\left(\frac{2\pi}{x}\right)^{\frac{1}{k}} \nonumber \\
& \times \sum_{j=0}^{\frac{k}{2}-1}  \Bigg\{  e^{\frac{i\pi(2j+1)(k-1)}{2k}} \; \bar{L}_{k}\left(2\pi \left(\frac{2\pi}{x}\right)^{\frac{1}{k}} e^{-\frac{i \pi (2j+1))}{2k}}  \right) \nonumber \\
&+ e^{-\frac{i\pi(2j+1)(k-1)}{2k}} \; \bar{L}_{k}\left(2\pi \left(\frac{2\pi}{x}\right)^{\frac{1}{k}} e^{\frac{i \pi (2j+1)}{2k}} \right)  \Bigg\},  \label{Wigert}
\end{align}
where 
\begin{align*}
\bar{L}_k(x):= \sum_{n=1}^\infty \frac{n^{\frac{1}{k}-1 }}{\exp(x n^{\frac{1}{k}} )-1}.
\end{align*}
\end{corollary}
This formula has been generalized by Dixit and the second author \cite[Theorem 1.5]{DM20} and further a two-variable generalization obtained by Dixit, Gupta,  Kumar,  and Maji  \cite[Theorem 2.12]{DGKM20}. 



\begin{theorem}\label{k even r odd}

Let $k \geq 2$ be an even integer and $r \neq -1$ be an odd integer.  Let $D_{k,r}(n)$ and $S_{k,r}(n)$ be defined as in \eqref{definition of D_{k,r}}.  Then for any $x>0$,  we have
 \begin{align*}
 \sum_{n=1}^{\infty} D_{k,r}(n) e^{-nx} & =-\frac{1}{2}\zeta(-r)+\frac{1}{k}\Gamma\left(\frac{1}{k}\right)\zeta\left(\frac{1}{k}-r\right)x^{-\frac{1}{k}}+R_{1+r} +\frac{(-1)^{\frac{2k+r-1}{2}} (2\pi)^{\frac{k+1-2r}{2}}}{x\,k^{\frac{2k-1}{2}}} 
 \\
  & \times  \sum_{j=-(k-1)}^{(k-1)} {}^{''} i^j \sum_{n=1}^{\infty} S_{k,r}(n) G_{0,k}^{\,k,0} \!\left(  \,\begin{matrix}\{\}\\r,-\frac{1}{k},\cdots ,-\frac{(k-1)}{k} \end{matrix} \; \Big|X(j)  \right),
  \end{align*}
 where $X(j)$ and $ R_{1+r}$ are defined as in \eqref{X(j)} and \eqref{R(1+r)} respectively. 
\end{theorem}
Substituting $k=2$ and $r=1$ in the above result,  we derive an interesting identity for $\zeta\left(- \frac{1}{2} \right)$. 

\begin{corollary}\label{Formula for zeta(-1/2)}

For $\alpha,\beta>0$ such that $\alpha\beta=4 \pi^3$,
\begin{align*}
\sum_{n=1}^{\infty}\frac{1}{n^{2}}\frac{\mathrm{d}}{\mathrm{d\alpha}}\left(\frac{1}{1-e^{n^{2}\alpha}}\right)
&=\frac{1}{24}+\frac{\sqrt{\beta}}{4 \pi}\zeta\left( -\frac{1}{2} \right)+\frac{\pi^4}{90 \alpha^2}\\ 
& -\frac{\pi}{4 \alpha} \sum_{m=1}^{\infty}\frac{1}{m}\Bigg\{ \left( \frac{\cos(\sqrt{m\beta})+\sin(\sqrt{m\beta})-e^{-\sqrt{m\beta}}}{ \sqrt{m \beta} \left(\cosh(\sqrt{m\beta})-\cos(\sqrt{m\beta})  \right)}\right) \\
&+ \frac{ 2 \sin(\sqrt{m\beta})\sinh(\sqrt{m\beta})}{(\cos({m\beta})\cosh(\sqrt{m\beta})-1)^2+(\sin(\sqrt{m\beta})\sinh(\sqrt{m\beta}))^2}\Bigg\}. 
\end{align*}
\end{corollary}

Note that, Theorem \ref{k even r odd} does not hold for $r=-1$,  so we present the next result corresponding to $r=-1$.
\begin{theorem}\label{special case k even r=-1 }
Let $k \geq 2$ be an even integer. Then for $x>0$, we have
 \begin{align*}
 \sum_{n=1}^{\infty} D_{k,-1}(n) e^{-nx}=& \frac{1}{2} \log\left( \frac{x}{(2\pi)^k} \right)+\frac{1}{k}\Gamma\left(\frac{1}{k}\right)\zeta\left(1+\frac{1}{k}\right)x^{-\frac{1}{k}} \\
 & + (-1)^k \sum_{j=-(k-1)}^{(k-1)} {}^{''} e^{ i \pi j} \sum_{m=1}^\infty \log \left[1- \exp\left(- e^{-\frac{i\pi j}{2k}} (2\pi)^{1+\frac{1}{k}} \left( \frac{m}{x} \right)^{\frac{1}{k} }  \right)  \right].
 \end{align*}
\end{theorem}

Till now, we have stated identities when $k \geq 2$ is even.  From the next result onwards we dealt with  $k\geq 1$ odd.  It turns out that,  Ramanujan's formula for $\zeta(2m+1)$ can be derived as a special case of our result when $k$ is an odd integer.

\begin{theorem}\label{k odd r odd}
Let $k\geq 1$ be an odd integer and $r \neq -1$ be an odd integer.  Let $D_{k,r}(n)$ and $S_{k,r}(n)$ be defined as in \eqref{definition of D_{k,r}}.  Then for $x> 0$,  we have
\begin{align}\label{Rama_generalization}
 \sum_{n=1}^{\infty} D_{k,r}(n) e^{-nx}=&-\frac{1}{2}\zeta(-r)+\frac{1}{k}\Gamma\left(\frac{1}{k}\right)\zeta\left(\frac{1}{k}-r\right)x^{-\frac{1}{k}}+R_{1+r}+R+\frac{(-1)^{\frac{2k+r-1}{2}} (2\pi)^{\frac{k+1-2r}{2}}}{x\, k^{\frac{2k-1}{2}}} \nonumber \\
 & \times
  \sum_{j=-(k-1)}^{(k-1)} {}^{''} i^j \sum_{n=1}^{\infty} S_{k,r}(n)G_{0,k}^{\,k,0} \!\left(  \,\begin{matrix}\{\}\\r,-\frac{1}{k},\cdots ,-\frac{(k-1)}{k} \end{matrix} \; \Big| X(-j)\right),
 \end{align}
where $X(j)$ and $R_{1+r}$ are defined as in \eqref{X(j)} and \eqref{R(1+r)} respectively,  and
\begin{align*}
R =
  \begin{cases}
   0, & {\rm{ if}}\,\, r \geq 1,  \\
  \frac{(-1)^{ \frac{1+r}{2} }}{2} \sum_{i=0}^{\frac{-(1+r)}{2}}  \frac{(-1)^{i+1} B_{k(2i+1)+1}}{(2i+1)! (k(2i+1)+1)  }  \frac{  B_{-2i-1-r}  (2\pi)^{-r  }}{ (-(2i+1+r))!} \left( \frac{x}{2\pi} \right)^{2i+1}, & {\rm if} \,\,  r \leq -3. 
 \end{cases}
\end{align*}


\end{theorem}
The next result provides a relation between two odd zeta values, namely, $\zeta(2m+1)$ and $\zeta(2km+1)$,  and a zeta value at rational argument $\zeta \left( \frac{1}{k}+2m+1 \right)$. 
\begin{corollary}\label{k odd r -ve odd}
Let $k\geq 1$ be an odd integer and $m$ be a positive integer.  For any $x>0$,  we have
\begin{align*}
\sum_{n=1}^{\infty} D_{k,-(2m+1)}(n) e^{-nx} & =- \frac{1}{2}\zeta(2m+1) + \frac{1}{k}\Gamma\left(\frac{1}{k}\right)\zeta\left(\frac{1}{k}+2m+1\right)x^{-\frac{1}{k}} \nonumber \\
& +  \frac{ (-1)^m}{2} \frac{ k (2km)! \zeta(2km+1) }{ (2m)!} \left( \frac{x}{(2\pi)^k} \right)^{2m} \nonumber \\ 
& + \frac{(-1)^m (2\pi)^{2m+1}}{2} \sum_{i=0}^m \frac{(-1)^{i+1} B_{k(2i+1)+1} B_{2m-2i} }{(2i+1)! (k(2i+1)+1) (2m-2i)!} \left( \frac{x}{2\pi} \right)^{2i+1} \nonumber \\
& + \frac{(-1)^m (2\pi)^{\frac{k+1}{2}+2m+1 }}{x \, k^{k-\frac{1}{2}}} \sum_{j=-(k-1)}^{(k-1)} {}^{''} i^j \sum_{n=1}^{\infty} S_{k,-(2m+1)}(n) \nonumber \\
& \qquad \qquad \qquad \qquad \times G_{0,k}^{\,k,0} \!\left(  \,\begin{matrix}\{\}\\ -(2m+1),-\frac{1}{k},\cdots ,-\frac{(k-1)}{k} \end{matrix} \; \Big| X(-j)\right).
\end{align*}
\end{corollary}
Substituting $k=1$ in Corollary \ref{k odd r -ve odd},  and simplifying, we derive Ramanujan's formula
for odd zeta values. 
\begin{corollary}\label{Ramanujan_odd zeta}
Let $m$ be a positive integer.  For any $x>0$,  
\begin{align}\label{Rama_odd zeta}
 \sum_{n=1}^{\infty} \sigma_{-(2m+1)}(n) e^{-nx} + \frac{1}{2} \zeta(2m+1) & =
    (-1)^m\left( \frac{x}{2\pi} \right)^{2m}  \left( \frac{ \zeta(2m+1)}{2}  +  \sum_{n=1}^{\infty}  \sigma_{-(2m+1)}(n)  e^{-\frac{4 \pi^2 n}{x}} \right) \nonumber \\ 
& + \frac{1}{2} \sum_{i=0}^{m+1} \frac{(-1)^{i+1} B_{2i} B_{2m+2-2i} x^{2m+1} }{(2i)! (2m+2-2i)!} \left( \frac{2 \pi}{x} \right)^{2i}.
\end{align}
\end{corollary}
One can easily show that \eqref{Rama_odd zeta} is equivalent to \eqref{Ramanujan's formula} under the substitution $x=2\alpha$ and $\alpha \beta= \pi^2$. 
Again,  letting $k=1$ and $r=2m+1$,  for $m\geq 1$,  in Theorem \ref{k odd r odd},  one can obtain the following interesting identity. 
\begin{corollary}\label{special cases when m less than -1}
Let $m$ be a positive integer.  For $\alpha,\beta>0$ with $\alpha \beta=\pi^2$,  we have
\begin{align}\label{Rama_k=1_ r=2m+1}
\alpha^{m+1} \sum_{n=1}^{\infty} \frac{n^{2m+1}}{e^{2n\alpha}-1}-(-\beta)^{m+1}  \sum_{n=1}^{\infty} \frac{n^{2m+1}}{e^{2n\beta}-1}=(\alpha^{m+1} -(-\beta)^{m+1}) \frac{B_{2m+2}}{4m+4}.
\end{align}
\end{corollary}
This formula is a special case of \eqref{Ramanujan's formula} and it can be found in \cite[Vol.~1, p.~259]{ramnote}.


Note that Theorem {\rm \ref{k odd r odd}} is not valid for $r=-1$.  Corresponding to $r=-1$ and $k\geq 1$ odd,  we obtain the next result.

\begin{theorem}\label{k odd r=-1}
Let $k \geq 1$ be an odd integer. Then for $x>0$, we have
 \begin{align}
 \sum_{n=1}^{\infty} D_{k,-1}(n) e^{-nx}=& \frac{1}{2} \log\left( \frac{x}{(2\pi)^k} \right)+\frac{1}{k}\Gamma\left(\frac{1}{k}\right)\zeta\left(1+\frac{1}{k}\right)x^{-\frac{1}{k}} + \frac{\zeta(-k)x}{2} \nonumber\\
 & - \sum_{j=-(k-1)}^{(k-1)} {}^{''}  \sum_{m=1}^\infty \log \left[1- \exp\left(- e^{\frac{i\pi j}{2k}} (2\pi)^{1+\frac{1}{k}} \left( \frac{m}{x} \right)^{\frac{1}{k} }  \right)  \right]. \label{gen_dekekind}
 \end{align}
\end{theorem}
Plugging $k=1$ in the above identity,  we obtain an interesting identity. 

\begin{corollary}\label{formula for logarithm Dedekind eta function}
For $\alpha ,\beta >0$ and $\alpha\beta=\pi^2 $,
\begin{equation}\label{Dedekind}
\sum_{n=1}^{\infty}\frac{1}{n(e^{2n\alpha}-1)}-\sum_{n=1}^{\infty}\frac{1}{n(e^{2n\beta}-1)}=\frac{\beta-\alpha}{12}-\frac{1}{4}\log\left(\frac{\alpha}{\beta}\right).
\end{equation}
\begin{remark}
This identity is popularly known as the transformation formula for the logarithm of the Dedekind eta function $\eta(z)$,  defined as $e^{\frac{\pi i z}{12}} \prod_{n=1}^\infty (1- e^{2n \pi i z})$,  which is an important example of a half-integral weight modular form.  Ramanujan recorded this identity twice in his second notebook \cite[Ch. 14, Sec. 8, Cor. (ii) and Ch. 16, Entry 27(iii)]{ramnote}.  For more information,  one can see \cite[p.~256]{bcbramsecnote}, \cite[p.~43]{bcbramthinote}, \cite[p.~320, Formula (29)]{lnb}. 
\end{remark}
\end{corollary}

In the next chapter we mention a few important results which will be useful throughout the manuscript.  We mention the proof of those results which are not well-known. 
\section{Well-known results}
An asymmetric form of the functional equation of $\zeta(s)$ can be written as
\begin{align}\label{Functional equation in asymmetric form}
\zeta(s)=2^{s}\pi^{s-1}\zeta(1-s)\Gamma(1-s)\sin\left(\frac{\pi s}{2}\right).
\end{align}
We know that $\zeta(s)$ has a simple pole at $s=1$ with residue $1$ and the Laurent series expansion is 
\begin{align}\label{Series expansion of Zeta}
\zeta(s)=\frac{1}{s-1}+\gamma+ \sum_{n=1}^{\infty} \frac{(-1)^n \gamma_n (s-1)^n}{n!},
\end{align}
where $\gamma_n$ are called the Stieltjes constants and defined by
\begin{align*}
\gamma_n= \lim_{i \rightarrow \infty} \left( \sum_{j=1}^{i} \frac{(\log j)^n}{j} -\frac{( \log i)^{n+1}}{n+1} \right),
\end{align*}
and 
 $\gamma$ is the well-known Euler–Mascheroni constant.    
The Laurent series expansion of $\Gamma(s)$ at $s=0$ is 
\begin{align}\label{series expansion of Gamma around zero}
 \Gamma(s)=\frac{1}{s}-\gamma+\frac{1}{2}\left(\gamma^2 +\frac{\pi^2}{6}\right)s-\frac{1}{6}\left(\gamma^3+\gamma \frac{\pi^2}{2}+2\zeta(3) \right)s^2+\cdots.
\end{align}
More generally,  $\Gamma(s)$ has Laurent series expansion at every negative integer.  
\begin{lemma}\label{Laurent_gamma_1+r}
For $r \leq -2$,  we have
$$ \Gamma(s)=\frac{a_{-1}}{s-(1+r)}+a_{0}+a_{1}(s-(1+r))+\cdots, $$
where 
$a_{-1}=\frac{(-1)^{r+1}}{(-(1+r))!} $
and $ a_{0}=\frac{(-1)^{r+1}}{(-(1+r))!} \left(-\gamma +\sum_{k=1}^{-(1+r)}\frac{1}{k}\right)$.
\end{lemma}
Euler proved that $\Gamma(s)$ satisfy the following beautiful reflection formula:
\begin{align}\label{Euler reflection formula}
\Gamma(s)\Gamma(1-s)=\frac{\pi}{\sin(\pi s)}, \quad  \forall  s\in \mathbb{C-Z}.
\end{align}
The multiplication formula for the gamma function stated as
\begin{lemma}\label{Multiplication formula for Gamma function}
For any positive integer $k$, 
\begin{align}\label{multiplication}
\Gamma(ks)=\frac{k^{ks}}{\sqrt{k}(2\pi)^{\frac{k-1}{2}}}\prod_{l=1}^{k-1}\Gamma(s)\Gamma\left(s+\frac{l}{k}\right).
\end{align}
In particular,  substituting $k=2$,  we have
\begin{align*}
\Gamma(2s)= \frac{2^{2s-1}}{\pi} \Gamma(s) \Gamma\left( s+ \frac{1}{2} \right).
\end{align*}
This is known as duplication formula. 
\end{lemma}
The next result gives an important information about the asymptotic expansion of the gamma function,  popularly known as Stirling's formula. 
\begin{lemma}\label{Stirling}
 For $s = \sigma + i T$ in a vertical strip $\alpha \leq \sigma \leq \beta$, 
\begin{equation}\label{stirling equn}
|\Gamma (\sigma + i T)| = \sqrt{2\pi} | T|^{\sigma - 1/2} e^{-\frac{1}{2} \pi |T|} \left(1 + O\left(\frac{1}{|T|}\right)  \right),
\end{equation}
as $|T|\rightarrow \infty$. 
 \end{lemma}

\begin{proof}
The proof of this result can be found in \cite[p. ~151]{IK}.
\end{proof} 

\begin{lemma}\label{boundzeta}
For  $\sigma\geq \sigma_0$, 
 there exist a  
constant $C(\sigma_0)$, such that 
\begin{equation*}
|\zeta(\sigma + iT) | \ll |T|^{C(\sigma_0)} 
\end{equation*}
as $|T| \rightarrow \infty.$
\end{lemma}

\begin{proof}
One can find the proof in \cite[p.~95]{Tit}.
\end{proof}

%

Next, we shall mention a few important special cases of the Meijer $G$-function. 
An immediate special case is the following result.  Putting $ m=q=1$ and $n=p=0$ in \eqref{MeijerG},  one can see that
\begin{align}\label{MeijerG(1,0,0,1)}
 G_{0,1}^{\,1,0} \!\left(  \,\begin{matrix}\{\}\\b \end{matrix} \; \Big|z  \right)=  \frac{1}{2 \pi i} \int_{c-i \infty}^{c+i \infty} \Gamma(b+s) z^{-s} \mathrm{d}s= e^{-z} z^b,
 \end{align}
 provided $\Re(s)=c > -\Re(b)$ and $\Re(z)>0$. 
Again, plugging $m=q=2$ and $n=p=0$ in the definition \eqref{MeijerG} of the Meijer $G$-function,  we will have the following result.

\begin{lemma}\label{MeijerG(2,0,0,2)}
Assuming the real part of the line of integration is bigger than $\max(-Re(b_1), - \Re(b_2)).  $ For $|\arg(z)|< \pi$,  we have
\begin{align*}
G_{0,2}^{\,2,0} \!\left(  \,\begin{matrix}\{\}\\b_1,b_2 \end{matrix} \; \Big| z   \right)=2 z^{\frac{1}{2} (b_1+b_2)}K_{b_1-b_2}(2\sqrt{z}),
\end{align*}
where $K_\nu(z)$ is the modified Bessel function.  
\end{lemma}
\begin{proof}
This can be verified from \cite[p.~115, Equn (11.1)]{Ober}.
\end{proof}

\begin{lemma} \label{MeijerG(k,0,0,k)}
Let $k \geq 2$ be a positive integer.  Then 
\begin{align*}
G_{0,k}^{\,k,0} \!\left(  \,\begin{matrix}\{\}\\b,b+\frac{1}{k},b+\frac{2}{k},\cdots ,b+\frac{(k-1)}{k} \end{matrix} \; \Big|z  \right)=\frac{(2\pi)^{\frac{k-1}{2}}}{\sqrt{k}} z^b e^{-kz^{1/k}},
\end{align*}
provided $|\arg(z)| < \frac{k\pi}{2}. $
\end{lemma}

\begin{proof}
With the help of the definition \eqref{MeijerG} of the Meijer-$G$ function,  we write
\begin{align*}
G_{0,k}^{\,k,0} \!\left(  \,\begin{matrix}\{\}\\ b,b+\frac{1}{k},\cdots ,b+\frac{(k-1)}{k} \end{matrix} \; \Big|z \right) &=\frac{1}{2\pi i} \int_{L}\prod_{l=1}^{k}\Gamma(s+b) \Gamma\left(s+b+\frac{l}{k}\right)z^{-s} \textrm{d}s,
\end{align*}
provided the poles of each of these gamma factor lie on the left of the line of integration.
Replacing $s$ by $s+b$ in the multiplication formula \eqref{multiplication} for $\Gamma(s)$,  we have
\begin{align*}
\Gamma(k(s+b))=\frac{k^{k(s+b)}}{\sqrt{k}(2\pi)^{\frac{k-1}{2}}}\prod_{l=1}^{k-1}\Gamma(s+b)\Gamma\left(s+b+\frac{l}{k}\right).
\end{align*}
Plugging this expression,  one can see that
\begin{align*}
G_{0,k}^{\,k,0} \!\left(  \,\begin{matrix}\{\}\\ b,b+\frac{1}{k},\cdots ,b+\frac{(k-1)}{k} \end{matrix} \; \Big|z \right) & = \sqrt{k}(2\pi)^{\frac{k-1}{2}} \frac{1}{ 2\pi i} \int_{L} \frac{\Gamma(k(s+b))}{k^{k(s+b)} }z^{-s} \textrm{d}s \\
&=\sqrt{k} (2 \pi)^{\frac{k-1}{2}}  \frac{1}{2\pi i} \int_{L'}\Gamma(s_{1}) z^{-\frac{s_{1}}{k}+b} \frac{\textrm{d}s_1}{k^{s_{1}+1}}
\\&=\frac{(2 \pi)^{\frac{k-1}{2}} }{\sqrt{k}} \frac{1}{2\pi i} \int_{L'}\Gamma(s_{1})(k z^{\frac{1}{k}})^{-s_{1}} z^b {\rm d}s_{1}\\
&=\frac{(2 \pi)^{\frac{k-1}{2}} z^b }{ \sqrt{k}} e^{-kz^{\frac{1}{k}}},
\end{align*}
the last step is possible only when $|\arg(z)| < \frac{k \pi}{2}$. 
\end{proof}

Dixit and Maji \cite[Lemma 3.1] {DM20} used the below lemma to prove \eqref{zetageneqn}.
\begin{lemma}\label{dixit_maji_lemma 1}
Let $a, u,v$ be three real numbers. Then
\begin{align*}
2 \Re\left(\frac{e^{iuv}}{\exp(ae^{-iu})-1}\right)=\frac{\cos(a\sin(u)+uv)-e^{-a \cos(u)}\cos(uv)}{\cosh(a\cos(u))-\cos(a\sin(u))}.
\end{align*}
\end{lemma}
In a similar vein,  we have the following lemma.
\begin{lemma}\label{our lemma 1}
Let $a,u,v$ be three real numbers. Then
\begin{align*}
2 \Re\left(\frac{ie^{iuv}}{\exp(ae^{-iu})-1}\right)
=\frac{-\sin(uv+a\sin(u))+e^{-a\cos(u)}\sin(uv)}{\cosh(a\cos(u))-\cos(a\sin(u))}.
\end{align*}
\end{lemma}
\begin{proof}
The left hand side expression can be written as
\begin{align*}
2 \Re\left(\frac{ie^{iuv}}{ \exp(ae^{-iu})-1}\right)&=2 \Re\left(\frac{-\sin(uv)+i\cos(uv)}{e^{a\cos(u)}(\cos(a\sin(u))-i\sin(a\sin(u)))-1}\right),
\end{align*}
now multiplying the numerator and the denominator by the conjugate of the denominator reduces to 
\begin{align*}
& 2 \Re\left(\frac{(-\sin(uv)+i\cos(uv))({e^{a\cos(u)}\cos(a\sin(u))+ie^{a\cos(u)}\sin(a\sin(u)))-1})}{e^{2a\cos(u)}-2e^{a\cos(u)}(\cos(a\sin(u)))+1}\right) \\
&=2\left(\frac{-e^{a\cos(u)}(\cos(a\sin(u))\sin(uv))+\sin(uv)-e^{a\cos(u)}(\sin(a\sin(u))\cos(uv))}{e^{2a\cos(u)}-2e^{a\cos(u)}(\cos(a\sin(u)))+1}\right)\\
&=2\left(\frac{-e^{a\cos(u)}(\sin(uv+a\sin(u))+\sin(uv)}{e^{2a\cos(u)}-2e^{a\cos(u)}(\cos(a\sin(u)))+1}\right), \\
&  =\frac{-\sin(uv+a\sin(u))+e^{-a\cos(u)}\sin(uv)}{\cosh(a\cos(u))-\cos(a\sin(u))}.
\end{align*}
In the final step we divided by $e^{a\cos(u)}$ on the numerator as well as on the denominator.

\end{proof}

The next lemma will play a crucial role to obtain our main results.

\begin{lemma}\label{use in k even r even}
Let $s \in \mathbb{C}$ and $k\in \mathbb{N}$.  Then
\begin{align*}
 \frac{\sin(ks)}{\sin(s)}=\sum_{j=-(k-1)}^{(k-1)} {}^{''} \exp(ijs),
\end{align*}
where $''$ means summation runs over $ j=-(k-1),-(k-3), \cdots ,(k-3),(k-1)$.
Thus,  for $k$ even, 
\begin{equation}\label{use in k even r odd}
\frac{\sin(ks)}{\cos(s)}=(-1)^{\frac{k}{2}}\sum_{j=-(k-1)}^{(k-1)} {}^{''} \, i^j \exp (ijs),
\end{equation}
and for $k$ odd,
\begin{equation}\label{use in k odd r odd}
\frac{\cos(ks)}{\cos(s)}=(-1)^{\frac{k-1}{2}}\sum_{j=-(k-1)}^{(k-1)} {}^{''} \, i^j \exp (-ijs). 
\end{equation}
\end{lemma}

\section{Proof of main results}

\begin{proof}[Theorem {\rm \ref{k even r even}}][]
It is well-known that $e^{-x}$ is the inverse Mellin transform of the gamma function $\Gamma(s)$,  that is,  for any $x>0$, 
\begin{align}\label{inverse_Mellin_gamma}
e^{-x} = \frac{1}{2 \pi i } \int_{c-i \infty}^{c+ i \infty} \Gamma(s) x^{-s} \mathrm{d}s,
\end{align}
valid for any $c>0$. 
Replacing $x$ by $nx$ in \eqref{inverse_Mellin_gamma} and using \eqref{gen_D(k,r)}, one can easily show that
\begin{align}\label{infinite sum_interms_line integral}
 \sum_{n=1}^{\infty} D_{k,r}(n) e^{-nx}  = \frac{1}{2\pi i}\int_{c-i\infty}^{c+i\infty} \Gamma(s)\zeta(ks)\zeta(s-r)x^{-s}\textrm{d}s,
\end{align}
provided with $\Re(s)=c > \max( \frac{1}{k}, 1+r )$.
In order to evaluate this line integral we choose a contour $ \mathcal{C}$ in a way so that all the poles of the integrand function lie inside the contour $ \mathcal{C}$, 
 determined by the line segments $[c -i T, c + i T] ,[c+iT ,-\lambda+iT],  [-\lambda+ i T, -\lambda - i T]$,   and $[-\lambda - i T, c - i T]$,  where we wisely take $\lambda>\max(0,-r)$ and $T$ is some large positive real number. The essence of this choice of $\lambda$ is justified at the later stage in our proof.
 Now applying Cauchy's residue theorem, we have 
\begin{align}\label{Cauchy Residue Theorem}
 \frac{1}{2\pi i}  \int_{ c - i T }^{c + i T} +  \int_{ c + i T }^{-\lambda + i T} +
\int_{ -\lambda + i T }^{-\lambda - i T} +
 \int_{ -\lambda - i T }^{c - i T} 
& \Gamma(s) \zeta(ks) \zeta(s-r)  x^{-s}  {\rm d}s \nonumber \\
&= \sum_{\rho} \underset{s=\rho}{\Res} \, \Gamma(s) \zeta(ks) \zeta(s-r)  x^{-s},
\end{align}
where the sum over $\rho$ runs through all the poles of $\Gamma(s) \zeta(ks) \zeta(s-r)  x^{-s}$ inside the contour $\mathcal{C}$.

First, let us analyse the poles of the integrand function. 
We know that $\Gamma(s)$ has simple poles at $s=0$ and negative integers.  The poles at negative integers are getting cancelled by the trivial zeros of $\zeta(ks)$ as we are dealing with $k$ even.  Thus,  the only simple poles of the integrand function $\Gamma(s) \zeta(ks) \zeta(s-r)x^{-s}$ are at $s=0,  \frac{1}{k}$ and $1+r$,  due to the simple pole of $\zeta(ks)$ at $s=1/k$ and the pole of $\zeta(s-r)$ at $s=1+r$.
Here we would like to mention that,  as we have considered $r$ is any even integer,  so it may happen that $1+r$ is a negative integer,  and in that case,  it seems that $1+r$ is a pole of order $2$ due to the factor $\Gamma(s)\zeta(s-r)$, but $\zeta(ks)$ will have a  zero at $1+r$ since $k$ is even.  This shows that $1+r$ remains a simple pole of the integrand function as we are dealing with $k\geq 2$ is even and $r$ is even.  

One can easily evaluate residues at $s=0$ and $s= \frac{1}{k}$.  Let $R_{\rho}$ denotes the residue term at $\rho$.  Using residue calculation methods,  we have
\begin{align}\label{R_0}
R_{0} =\lim_{s \rightarrow 0} s \Gamma(s)\zeta(ks)\zeta(s-r)x^{-s} 
=\zeta(0) \zeta(-r)
 = - \frac{1}{2} \zeta(-r).
\end{align}
Similarly,  we obtain
\begin{align}\label{Residue at 1/k}
 R_{\frac{1}{k}}=
 \lim_{s \rightarrow \frac{1}{k}} \left(s-\frac{1}{k} \right) \Gamma(s)\zeta(ks)\zeta(s-r)x^{-s}  
=\frac{1}{k}\Gamma\left(\frac{1}{k}\right)\zeta\left(\frac{1}{k} -r\right)x^{-\frac{1}{k}}.
\end{align}
Here we note that while calculating the residue at $s=1+r$,  we have to make two different cases. 
First,  if $r\geq 0$,  then
\begin{align}\label{R_1+r}
R_{1+r}=\lim_{s \rightarrow 1+r} (s-(1+r)) \Gamma(s)\zeta(ks)\zeta(s-r)x^{-s} 
&=\Gamma(1+r)\zeta(k(1+r)) x^{-(1+r)} \nonumber \\
&= r! \zeta(k(1+r))x^{-(1+r)}.
\end{align}
Second,  if $r<0$ is a negative even integer,  then $1+r$ is also a negative integer. Thus,  $\Gamma(s)$ satisfy the Laurent series expansion \eqref{Laurent_gamma_1+r} at $1+r$.  Therefore,  we have
\begin{align}\label{residue_gamma_1+r}
\lim_{s \rightarrow 1+r} (s- 1-r) \Gamma(s) = \frac{(-1)^{1+r}}{(-(1+r))!}.
\end{align}
Again, $\zeta(ks)$ has a trivial zero at $s=1+r$ as $k$ is even,  so we have 
\begin{align}\label{residue_zeta(ks)_1+r}
\lim_{s \rightarrow 1+r} \frac{  \zeta(ks)}{(s- 1-r)} = k \zeta'(k(1+r)).
\end{align}
Combining \eqref{residue_gamma_1+r} and \eqref{residue_zeta(ks)_1+r} and together with the fact that $1+r$ is a simple pole of $\zeta(s-r)$,  we obtain
\begin{align}\label{Res_1+r}
R_{1+r}  = \lim_{s \rightarrow (1+r)}(s -( 1+r)) \Gamma(s) \zeta(ks) \zeta(s-r)x^{-s} 
= \frac{(-1)^{1+r}}{(-(1+r))!} k \zeta'(k(1+r)) x^{-(1+r)}.
\end{align}
We now proceed to show that the horizontal integrals 
 $$
 H_1(T,x):= \frac{1}{2\pi i} \int_{c+ i T }^{-\lambda + i T } \Gamma(s) \zeta(ks) \zeta(s -r) x^{-s} {\rm d}s 
 $$ 
 and 
 $$
 H_2(T, x)= \frac{1}{2\pi i} \int_{-\lambda- i T }^{c - i T } \Gamma(s) \zeta(ks) \zeta(s-r)x^{-s} {\rm d}s 
 $$
   vanish as $T \rightarrow \infty$.
  Replace $s$ by $\sigma+iT$ in $H_1(T,x)$ to see
    \begin{align*}
    |H_1(T,x) | & =\left|\frac{1}{2\pi i} \int_{c}^{d} \Gamma(\sigma+iT) \zeta(k\sigma+ikT) \zeta(\sigma+iT-r) x^{-\sigma+iT} {\rm d}\sigma \right| \\
    &  \ll  \int_{c}^{d}|\Gamma(\sigma+iT)\zeta(k \sigma+ikT) \zeta(\sigma+iT-r) |x^{-\sigma} {\rm d}\sigma \\
    & \ll |T|^A \exp\left( - \frac{\pi}{2} |T| \right),
   \end{align*}
for some constant $A$.  In the final step, we used Stirling's formula for $\Gamma(s)$ i.e.,  Lemma \ref{Stirling} and bound for $\zeta(s)$ i.e.,  Lemma \ref{boundzeta}. 
 This immediately implies that $H_1(T, x)$ vanishes as $T \rightarrow \infty$.  Similarly,  one can show that $H_2(T, x)$ also vanishes as $T \rightarrow \infty$.  Now letting $T\rightarrow \infty$ in \eqref{Cauchy Residue Theorem} and collecting all the residual terms,  we have 
\begin{align}\label{application_CRT}
 \frac{1}{2\pi i}  \int_{ (c)  } 
\Gamma(s) \zeta(ks) \zeta(s-r) x^{-s}  {\rm d}s 
 & = R_{0} +R_{\frac{1}{k}} + R_{1+r}  \nonumber \\
 & +\frac{1}{2\pi i} \int_{ (-\lambda) }
\Gamma(s) \zeta(ks) \zeta(s-r) x^{-s}  {\rm d}s.
\end{align}
Here and throughout the paper we denote $\int_{(c)}$ by $\int_{c-i\infty}^{c+i\infty} $.
Now one of our main objectives is to evaluate the following vertical integral 
\begin{align}\label{Vertical V(x)}
V_{k,r}(x):= \int_{ (-\lambda )} 
\Gamma(s) \zeta(ks) \zeta(s-r) x^{-s}  {\rm d}s. 
\end{align}
Utilize the asymmetric form \eqref{Functional equation in asymmetric form} of the functional equation of $\zeta(s)$ to write
\begin{align}
\zeta(ks) & =2^{ks} \pi^{ks-1}\zeta(1-ks)\Gamma(1-ks) \sin\left(\frac{\pi k s}{2} \right),  \label{zeta(ks)}\\
\zeta(s-r) & =2^{s-r} \pi^{s-r-1}\zeta(1-s+r)\Gamma(1-s+r) \sin \left(\frac{\pi(s-r)}{2} \right).\label{zeta(s-r)}
\end{align}
Substituting \eqref{zeta(ks)} and \eqref{zeta(s-r)}  in \eqref{Vertical V(x)},  the vertical integral becomes
\begin{align*}
V_{k,r}(x)= \frac{1}{(2\pi)^{r} \pi^{2}} \frac{1}{2\pi i}\int_{ (-\lambda)} &  \Gamma(s)\Gamma(1-ks)\Gamma(1-s+r)\zeta(1-ks)\zeta(1-s+r)\\
& \times \sin\left(\frac{\pi ks}{2}\right)\sin \left(\frac{\pi(s-r)}{2}\right)\left(\frac{(2\pi)^{k+1}}{x}\right)^{s} \textrm{d}s.
\end{align*}
Now to shift the line of integration,  we replace $s$ by $1-s$.  Then we have
\begin{align}\label{common expression for left vertical integral} 
 V_{k,r}(x)= \frac{(2\pi)^{k+1-r}}{ \pi^{2} x} \frac{1}{2\pi i}  & \int_{(1+\lambda)} \Gamma(1-s)\Gamma(1-k+ks)\Gamma(s+r)\zeta(1-k+ks)\zeta(s+r)\nonumber\\
\times &  \sin\left(\frac{\pi k(1-s)}{2}\right)\sin \left(\frac{\pi(1-s-r)}{2}\right)\left(\frac{(2\pi)^{k+1}}{x}\right)^{-s} \textrm{d}s.
\end{align}
Since $k$ and $ r$ are both even,  the following identities hold
\begin{align}
 \sin\left(\frac{\pi k(1-s)}{2}\right) &=(-1)^{\frac{k}{2}-1}\sin\left(\frac{ks\pi}{2}\right),  \label{sin_k_even}\\
\sin \left(\frac{\pi(1-s-r)}{2}\right) &= (-1)^{\frac{r}{2}} \cos\left( \frac{\pi s}{2} \right). \label{sin_r_even}
\end{align}
Using the above two expressions in \eqref{common expression for left vertical integral},  we see
\begin{align}
V_{k,r}(x)= \frac{(-1)^{\frac{k+r}{2}-1}(2\pi)^{k+1-r}}{ \pi^{2} x} \frac{1}{2\pi i} &   \int_{(1+\lambda)}  \Gamma(1-s)\Gamma(1-k+ks)\Gamma(s+r)\zeta(1-k+ks)\nonumber\\
\times & \zeta(s+r)  \sin\left(\frac{\pi k s}{2}\right)\cos \left(\frac{\pi s}{2}\right)\left(\frac{(2\pi)^{k+1}}{x}\right)^{-s} \textrm{d}s \nonumber \\
=\frac{(-1)^{\frac{k+r}{2}-1}(2\pi)^{k+1-r}}{ 2 \pi^{2} x} \frac{1}{2\pi i}  &   \int_{(1+\lambda)}  \Gamma(1-s)\Gamma(1-k+ks)\Gamma(s+r)\zeta(1-k+ks)\nonumber\\
\times & \zeta(s+r)  \frac{\sin\left(\frac{\pi k s}{2}\right)}{\sin \left(\frac{\pi s}{2}\right)} \sin(\pi s)\left(\frac{(2\pi)^{k+1}}{x}\right)^{-s} \textrm{d}s \nonumber \\
=\frac{(-1)^{\frac{k+r}{2}-1}(2\pi)^{k+1-r}}{ 2 \pi x} \frac{1}{2\pi i}  &  \int_{(1+\lambda)}   \frac{\Gamma(1-k+ks)\Gamma(s+r)}{ \Gamma(s)}\zeta(1-k+ks)\nonumber\\
\times & \zeta(s+r)  \frac{\sin\left(\frac{\pi k s}{2}\right)}{\sin \left(\frac{\pi s}{2}\right)} \left(\frac{(2\pi)^{k+1}}{x}\right)^{-s} \textrm{d}s. \label{V_k}
\end{align}
Here in the ultimate step, we have used Euler's reflection formula \eqref{Euler reflection formula} for $\Gamma(s)$ and in the penultimate step we multiplied by $\sin\left( \frac{\pi s}{2}  \right)$ in the denominator as well as in the numerator.  Now we notice that $\Re(s+r) = 1+\lambda+r >1$ and $\Re(1-k+ks)= 1+k \lambda >1$ as we have considered $\lambda >\max(0, -r)$.  This explains our choice of $\lambda$. Thus,  we can express $\zeta(1-k+ks)\zeta(s+r)$ as an infinite series,  mainly,  use  \eqref{gen_S(k,r)} and then interchange the summation and integration in \eqref{V_k} to deduce that
\begin{align}\label{V(k,r) in terms I(k,r)}
V_{k,r}(x) = \frac{(-1)^{\frac{k+r}{2}-1}(2\pi)^{k-r}}{ x} \sum_{n=1}^{\infty} S_{k,r}(n) I_{k,r}(n,x),
\end{align}
where
\begin{align}\label{I(k,r)}
I_{k,r}(n, x):=
\frac{1}{2\pi i}  \int_{(1+\lambda)}    \frac{\Gamma(1-k+ks)\Gamma(s+r)}{ \Gamma(s)}
   \frac{\sin\left(\frac{\pi k s}{2}\right)}{\sin \left(\frac{\pi s}{2}\right)} \left(\frac{(2\pi)^{k+1} n}{x}\right)^{-s} \textrm{d}s. 
\end{align}
Now our main goal is to simplify this integral.  Invoking Lemma \ref{use in k even r even} in \eqref{I(k,r)} we see that 
\begin{align}\label{Final I_k,r}
I_{k,r}(n, x)=  \sum_{j=-(k-1)}^{(k-1)} {}^{''}
\frac{1}{2\pi i}  \int_{(1+\lambda)}    \frac{\Gamma(1-k+ks)\Gamma(s+r)}{ \Gamma(s)} 
  \left(\frac{e^{-\frac{i \pi j}{2}} (2\pi)^{k+1} n}{x}\right)^{-s} \textrm{d}s. 
\end{align}
Employing multiplication formula \eqref{multiplication} for $\Gamma(s)$, one can derive that
\begin{align}\label{Use of multiplication}
\frac{\Gamma(1-k+ks)}{\Gamma(s)}= \frac{k^{ks}}{k^{k-\frac{1}{2}(2\pi )^{\frac{k-1}{2}}}} \prod_{l=1}^{k-1} \Gamma\left( s- \frac{l}{k} \right).
\end{align}
Substituting \eqref{Use of multiplication} in \eqref{Final I_k,r} yields that 
\begin{align}\label{I_(k,r) before MeijerG}
I_{k,r}(n, x)=  \frac{1}{k^{k-\frac{1}{2}} (2\pi )^{\frac{k-1}{2}}} \sum_{j=-(k-1)}^{(k-1)} {}^{''}
\frac{1}{2\pi i}  \int_{(1+\lambda)}    \Gamma(s+r) \prod_{l=1}^{k-1} \Gamma\left( s- \frac{l}{k} \right)
 X(j)^{-s} \textrm{d}s,
\end{align}
where 
\begin{align}\label{X(j)_defn}
 X(j)=\frac{e^{-\frac{i\pi j}{2}} (2\pi)^{k+1}n}{k^{k} x}
 \end{align} 
is exactly same as defined in \eqref{X(j)}.  Now we shall try to express, the line integral  in  \eqref{I_(k,r) before MeijerG},  in terms of the Meijer $G$-function and for that we have to verify all the necessary conditions for the convergence of the integral.  First,  one can check that all the poles of the integrand function lie on the left of the line of integration $\Re(s)=1+\lambda$ since $\lambda > \max(0, -r)$.  Now using the definition \eqref{MeijerG} of the Meijer $G$-function,  with $m=q=k$ and $n=p=0$ and $b_1=r,  b_2=-\frac{1}{k}, b_3= -\frac{2}{k}, \cdots, b_k= -\frac{k-1}{k}$,  we find that
\begin{align}\label{I_(k,r)_MeijerG}
\frac{1}{2\pi i}  \int_{(1+\lambda)}    \Gamma(s+r) \prod_{l=1}^{k-1} \Gamma\left( s- \frac{l}{k} \right)
 X(j)^{-s} \textrm{d}s= G_{0,k}^{\,k,0} \!\left(  \,\begin{matrix}\{ \}\\r, -\frac{1}{k},\cdots ,-\frac{(k-1)}{k} \end{matrix} \; \Big| X(j)  \right),
\end{align}
convergent since $p+q < 2(m+n)$ and $|\arg(X(j))|= |\frac{\pi j}{2}| < (m+n - \frac{p+q}{2})= \frac{\pi k}{2}$ as $|j|\leq k-1$.
Substituting \eqref{I_(k,r)_MeijerG} in \eqref{I_(k,r) before MeijerG}, we arrive at
\begin{align}\label{I_(k,r)_Final}
I_{k,r}(n, x)=  \frac{1}{k^{k-\frac{1}{2}} (2\pi )^{\frac{k-1}{2}}} \sum_{j=-(k-1)}^{(k-1)} {}^{''} G_{0,k}^{\,k,0} \!\left(  \,\begin{matrix} \{ \} \\r, -\frac{1}{k},\cdots ,-\frac{(k-1)}{k} \end{matrix} \; \Big| X(j)  \right).
\end{align}
Now plugging this final expression  \eqref{I_(k,r)_Final} of $I_{k,r}(n,x)$ in \eqref{V(k,r) in terms I(k,r)},  the left vertical integral becomes
\begin{align}\label{Final_left vertical}
V_{k,r}(x) = \frac{(-1)^{\frac{k+r}{2}-1}(2\pi)^{\frac{k+1}{2}-r}}{ k^{k-\frac{1}{2}} x}  \sum_{j=-(k-1)}^{(k-1)} {}^{''}   \sum_{n=1}^{\infty} S_{k,r}(n) G_{0,k}^{\,k,0} \!\left(  \,\begin{matrix} \{ \} \\r, -\frac{1}{k},\cdots ,-\frac{(k-1)}{k} \end{matrix} \; \Big| X(j)  \right).
\end{align}
At this moment,  we must show that the infinite series
\begin{align}\label{S_(k,r)_MeijerG}
 \sum_{n=1}^{\infty} S_{k,r}(n) G_{0,k}^{\,k,0} \!\left(  \,\begin{matrix} \{ \} \\r, -\frac{1}{k},\cdots ,-\frac{(k-1)}{k} \end{matrix} \; \Big| X(j)  \right)
\end{align}
is convergent for any fixed $k \geq 1,  r\in \mathbb{Z},  |j| \leq k-1$.  To show the convergence of this series, we shall use the integral representation \eqref{I_(k,r)_MeijerG} of the Meijer $G$-function. 
Employing Stirling's formula \eqref{stirling equn} on each gamma factor that is present in \eqref{I_(k,r)_MeijerG},  letting $s= 1+\lambda+iT$,  we have
\begin{align*}
|\Gamma(s+r)| & = |\Gamma(1+\lambda+r+iT)= O\left( |T|^{\lambda+r+\frac{1}{2}}e^{-\frac{\pi}{2}|T|} \right), \\
\left|\Gamma\left( s- \frac{l}{k} \right) \right| & = \left|\Gamma\left( 1+\lambda- \frac{l}{k} +iT\right) \right| = O\left( |T|^{\lambda- \frac{l}{k}+\frac{1}{2}}e^{-\frac{\pi}{2}|T|} \right), \,\, {\rm for}\,\, 1\leq l \leq k-1,
\end{align*}
as $T \rightarrow \infty$.  Using the definition \eqref{X(j)_defn} of $X(j)$,  we note that
\begin{align*}
|X(j)^{-s}| = \left| \left( \frac{e^{-\frac{i\pi j}{2}} (2\pi)^{k+1}n}{x k^{k}}\right)^{-1-\lambda-i T}  \right| =O_{k,x} \left( \frac{1}{n^{1+\lambda}} e^{ \frac{ \pi |j| |T|}{2}} \right).
\end{align*}
Utilizing the above bounds for gamma functions and the bound for $X(j)$,  and upon simplification,  we  derive that
\begin{align}\label{bound_MeijerG}
\left| G_{0,k}^{\,k,0} \!\left(  \,\begin{matrix} \{ \}\\r, -\frac{1}{k},\cdots ,-\frac{(k-1)}{k} \end{matrix} \; \Big| X(j)  \right) \right|  & \ll O_{k,x} \left(  \frac{1}{n^{1+\lambda}} \int_{-\infty}^{\infty} |T|^{k \lambda+ r +\frac{1}{2}} e^{ -\frac{\pi}{2} (k-|j|)|T|} \mathrm{d}T \right) \nonumber \\
& \ll O_{k,r,x} \left(  \frac{1}{n^{1+\lambda}} \right).
\end{align}
The integral present inside the above big-oh bound is convergent since $|j|\leq k-1$ and $k\lambda +r+ \frac{3}{2}$ is a positive quantity as we have $\lambda > \max(0, -r)$.  Plugging  \eqref{bound_MeijerG} in \eqref{S_(k,r)_MeijerG},  we can see that the infinite series \eqref{S_(k,r)_MeijerG} is convergent since the Dirichlet series $\sum_{n=1}^{\infty} S_{k,r}(n) n^{-(1+\lambda)}$ is convergent.

Finally,  combining \eqref{infinite sum_interms_line integral},  \eqref{application_CRT},  \eqref{Vertical V(x)}, and \eqref{Final_left vertical},  and together with all the residual terms \eqref{R_0}, \eqref{Residue at 1/k} \eqref{R_1+r},  and \eqref{Res_1+r},  we complete the proof of Theorem \ref{k even r even}. 

\end{proof}

\subsection{Recovering Ramanujan's formula for $\zeta(1/2)$}
\begin{proof}[Corollary {\rm \ref{zeta(1/2)}}][]
Substituting $k=2$ and $r=0$ in Theorem \ref{k even r even},  one can see that
\begin{align*}
\sum_{n=1}^{\infty} D_{2,0}(n) e^{-nx} & = \frac{1}{4}+\frac{\sqrt{\pi}}{2\sqrt{x}}\zeta\left(\frac{1}{2}\right)+\frac{\pi^2}{6 x}   \\
& + \frac{ \pi^{3/2}}{x} \sum_{n=1}^{\infty} S_{2,0}(n) \left[ G_{0,2}^{\,2,0} \!\left(  \,\begin{matrix} \{ \}\\0,-\frac{1}{2} \end{matrix} \; \Big| \frac{e^{-\frac{i \pi}{2}} 2n \pi^3}{x}   \right) + G_{0,2}^{\,2,0} \!\left(  \,\begin{matrix} \{ \}\\0,-\frac{1}{2} \end{matrix} \; \Big| \frac{e^{\frac{i \pi}{2}} 2n \pi^3}{x}   \right) \right]. 
\end{align*}
Observe that the Meijer $G$-functions that are present in the above equation are conjugate to each other.  Thus,  letting $x= \alpha$ and $\alpha \beta = 4 \pi^3$,  and using \eqref{D(k,r) for r positive},  we obtain
\begin{align}\label{Rama_zeta(1/2)}
\sum_{n=1}^{\infty}\frac{1}{e^{n^{2}\alpha}-1} & = \frac{1}{4}+\frac{\sqrt{\pi}}{2\sqrt{\alpha}}\zeta\left(\frac{1}{2}\right)+\frac{\pi^2}{6 \alpha} \nonumber \\
& + \frac{ 2 \pi^{3/2}}{\alpha} \sum_{n=1}^{\infty} S_{2,0}(n) \Re \left[ G_{0,2}^{\,2,0} \!\left(  \,\begin{matrix} \{ \}\\0,-\frac{1}{2} \end{matrix} \; \Big| \frac{e^{-\frac{i \pi}{2}} n \beta
}{2}   \right) \right].
\end{align}
Recall Lemma \ref{MeijerG(k,0,0,k)} with $k=2 $ and $b=-\frac{1}{2}$,  to produce
\begin{align}\label{applic_MeijerG}
G_{0,2}^{\,2,0} \!\left(  \,\begin{matrix} \{ \}\\0,-\frac{1}{2} \end{matrix} \; \Big| \frac{e^{-\frac{i \pi}{2}} n \beta
}{2}   \right) = \sqrt{\frac{2\pi}{n \beta}} e^{ \frac{i\pi}{4} } \exp \left(- \sqrt{2n\beta} e^{-\frac{i \pi}{4}}\right).
\end{align}
Now utilizing \eqref{applic_MeijerG} in \eqref{Rama_zeta(1/2)}, and using the definition \eqref{definition of D_{k,r}} of $S_{2,0}(n)$,   we see that
\begin{align}\label{before_Rama_zeta(1/2)}
\sum_{n=1}^{\infty}\frac{1}{e^{n^{2}\alpha}-1} & = \frac{1}{4}+\frac{\sqrt{\pi}}{2\sqrt{\alpha}}\zeta\left(\frac{1}{2}\right)+\frac{\pi^2}{6 \alpha} \nonumber \\
& + \sqrt{\frac{2 \pi}{\alpha}} \sum_{n=1}^{\infty} \sum_{d^2|n} \frac{d}{\sqrt{n}} \Re \left[ e^{ \frac{i\pi}{4} } \exp \left(- \sqrt{2n\beta} e^{-\frac{i \pi}{4}}\right) \right].
\end{align}
Next aim is to simplify the infinite series
\begin{align*}
\sum_{n=1}^{\infty} \sum_{d^2|n} \frac{d}{\sqrt{n}} \Re \left[ e^{ \frac{i\pi}{4} } \exp \left(- \sqrt{2n\beta} e^{-\frac{i \pi}{4}}\right) \right]. 
\end{align*}
Writing $n=d^2 m$,  one  can simplify the infinite sum in a following way
\begin{align}
& \sum_{n=1}^{\infty} \sum_{d^2|n} \frac{d}{\sqrt{n}} \Re \left[ e^{ \frac{i\pi}{4} } \exp \left(- \sqrt{2n\beta} e^{-\frac{i \pi}{4}}\right) \right] \nonumber \\
&= \sum_{m=1}^{\infty} \sum_{d=1}^{\infty} \frac{1}{\sqrt{m}} \Re \left[ e^{ \frac{i\pi}{4} } \exp \left(- \sqrt{2m\beta} d e^{-\frac{i \pi}{4}}\right) \right] \nonumber \\
&= \sum_{m=1}^{\infty}  \frac{1}{\sqrt{m}} \Re \left[  \frac{ e^{ \frac{i\pi}{4} }}{\exp \left( \sqrt{2m\beta}  e^{-\frac{i \pi}{4}}\right) -1 } \right] \nonumber \\
&= \sum_{m=1}^{\infty}  \frac{1}{ 2\sqrt{2m}}  \left( \frac{\cos(\sqrt{m\beta})-\sin(\sqrt{m\beta})-e^{-\sqrt{m\beta}}} { ( \cosh(\sqrt{m\beta})-\cos(\sqrt{m\beta}))} \right). \label{cos(mbeta)}
\end{align}
In the final step,  we have used Lemma \ref{dixit_maji_lemma 1} with $u=\pi/4,  v=1$, and $  a=\sqrt{2m\beta}$.  Now substituting \eqref{cos(mbeta)} in \eqref{before_Rama_zeta(1/2)},  one can finish the proof of \eqref{Ramanujan_zeta(1/2)}. 

 \end{proof}
 
\subsection{Wigert's formula for $\zeta(1/k)$ for $k \geq 2$ even.}
 
 \begin{proof}
First,  substituting $r=0$ in Theorem \ref{k even r even}, we find that
\begin{align}\label{r=0}
\sum_{n=1}^{\infty} D_{k,0}(n) e^{- nx} & = \frac{1}{4} + \frac{1}{k}\Gamma\left(\frac{1}{k}\right)\zeta\left(\frac{1}{k}\right)x^{-\frac{1}{k}} + \frac{ \zeta(k)}{x}  \nonumber \\
& + \frac{(-1)^{\frac{k-2}{2}} (2\pi)^{\frac{k+1}{2}}}{k^{\frac{2k-1}{2}} x} \sum_{j=-(k-1)}^{(k-1)} {}^{''} \sum_{n=1}^{\infty} S_{k,0}(n) \; G_{0,k}^{\,k,0} \!\left(  \,\begin{matrix}\{\}\\ 0,-\frac{1}{k},\cdots ,-\frac{(k-1)}{k} \end{matrix} \; \Big| X(j)   \right),
\end{align}
where $X(j)=\frac{e^{-\frac{i\pi j}{2}} (2\pi)^{k+1}n}{k^{k} x}.$
Now we shall try to simplify the last term and to do that we make use of Lemma \ref{MeijerG(k,0,0,k)}, with $b=-\frac{k-1}{k}$,  which gives rise to
\begin{align*}
 & \frac{(-1)^{\frac{k-2}{2}} (2\pi)^{\frac{k+1}{2}}}{k^{\frac{2k-1}{2}} x} G_{0,k}^{\,k,0} \!\left(  \,\begin{matrix}\{\}\\ 0,-\frac{1}{k},\cdots ,-\frac{(k-1)}{k} \end{matrix} \; \Big| X(j)   \right) \nonumber \\
 & = \frac{(-1)^{\frac{k}{2}-1}}{k} \left(  \frac{2\pi}{x} \right)^{\frac{1}{k}} e^{\frac{i \pi j(k-1)}{2k} } n^{\frac{1}{k}-1} \exp\left(- e^{-\frac{i \pi j}{2k}} (2 \pi)^{ \frac{1}{k} +1} \left( \frac{n}{x}\right)^{\frac{1}{k}}  \right).
\end{align*}
Implement this simplification of the Meijer $G$-function in the last term of \eqref{r=0}, and  use \eqref{D(k,r) for r positive} and the definition \eqref{definition of D_{k,r}} of $S_{k,0}(n)$,  then \eqref{r=0} becomes
\begin{align}\label{last step_Wigert}
\sum_{n=1}^{\infty}\frac{1}{e^{n^{k}x}-1} & = \frac{1}{4} + \frac{1}{k}\Gamma\left(\frac{1}{k}\right)\zeta\left(\frac{1}{k}\right)x^{-\frac{1}{k}} + \frac{ \zeta(k)}{x}  + \frac{(-1)^{\frac{k}{2}-1}}{k} \left(  \frac{2\pi}{x} \right)^{\frac{1}{k}} \nonumber \\
& \times  \sum_{j=-(k-1)}^{(k-1)} {}^{''} \sum_{n=1}^{\infty} \sum_{d^k|n} d^{k-1}  n^{\frac{1}{k}-1}  e^{\frac{i \pi j(k-1)}{2k} }\exp\left(- e^{-\frac{i \pi j}{2k}} (2 \pi)^{ \frac{1}{k} +1} \left( \frac{n}{x}\right)^{\frac{1}{k}}  \right).
\end{align}
Now writing $n= d^k m$ and upon simplification,  the last term reduces to
\begin{align}
&  \frac{(-1)^{\frac{k}{2}-1}}{k} \left(  \frac{2\pi}{x} \right)^{\frac{1}{k}} 
   \sum_{j=-(k-1)}^{(k-1)} {}^{''} \sum_{m=1}^{\infty} m^{\frac{1}{k}-1}   e^{\frac{i \pi j(k-1)}{2k} }  \sum_{d=1}^{\infty} \exp\left(- e^{-\frac{i \pi j}{2k}} (2 \pi)^{ \frac{1}{k} +1} \left( \frac{m}{x}\right)^{\frac{1}{k}} d  \right) \nonumber \\
 & =  \frac{(-1)^{\frac{k}{2}-1}}{k} \left(  \frac{2\pi}{x} \right)^{\frac{1}{k}} 
   \sum_{j=-(k-1)}^{(k-1)} {}^{''} e^{\frac{i \pi j(k-1)}{2k} }  \sum_{m=1}^{\infty}    \frac{m^{\frac{1}{k}-1} }{ \exp\left( e^{-\frac{i \pi j}{2k}} (2 \pi)^{ \frac{1}{k} +1} \left( \frac{m}{x}\right)^{\frac{1}{k}}   \right) -1} \nonumber  \\
   & = \frac{(-1)^{\frac{k}{2}-1}}{k} \left(  \frac{2\pi}{x} \right)^{\frac{1}{k}} \sum_{j=0}^{\frac{k}{2}-1} \Bigg[ e^{\frac{i \pi (2j+1)(k-1)}{2k} }  \sum_{m=1}^{\infty}    \frac{m^{\frac{1}{k}-1} }{ \exp\left( e^{-\frac{i \pi (2j+1)}{2k}} (2 \pi)^{ \frac{1}{k} +1} \left( \frac{m}{x}\right)^{\frac{1}{k}}   \right) -1} \nonumber  \\
 & \hspace{4cm} +  e^{-\frac{i \pi (2j+1)(k-1)}{2k} }  \sum_{m=1}^{\infty}    \frac{m^{\frac{1}{k}-1} }{ \exp\left( e^{\frac{i \pi (2j+1)}{2k}} (2 \pi)^{ \frac{1}{k} +1} \left( \frac{m}{x}\right)^{\frac{1}{k}}   \right) -1} \Bigg].  \label{last term}
\end{align}
Ultimately,  plugging \eqref{last term} in \eqref{last step_Wigert},  we complete the proof of Wigert's formula  \eqref{Wigert} for $\zeta(\frac{1}{k})$, $k \geq 2$ even. 

\end{proof}

 \begin{proof}[Theorem {\rm \ref{k even r odd}}][]
 The proof  goes in accordance with the proof of Theorem \ref{k even r even},  so we mention those steps where it differs from Theorem \ref{k even r even}.  Note that,  in this case,  we are dealing with $k\geq 2$ even and $r\neq -1$ is any odd integer.  
 First, we point out that,  the  analysis of poles will be exactly same as in Theorem \ref{k even r even}, so won't repeat it here.  The only changes will be in the calculation of the vertical integral of $V_{k,r}(x)$.  One can recall that,  while calculating $V_{k,r}(x)$,  the equation \eqref{common expression for left vertical integral} depends on $k$ and $r$.
   Since $k$ is even,  the equation \eqref{sin_k_even} will remain same.  For clarity,  we mention it once again.  Mainly,  the equation \eqref{sin_k_even} is
 \begin{align}
 \sin\left(\frac{\pi k(1-s)}{2}\right) &=(-1)^{\frac{k}{2}-1}\sin\left(\frac{ks\pi}{2}\right),  \label{sin_k_even_1}
 \end{align}
 whereas the equation \eqref{sin_r_even} changes to
 \begin{align}
\sin \left(\frac{\pi(1-s-r)}{2}\right) &= (-1)^{\frac{r+1}{2}} \sin\left( \frac{\pi s}{2} \right), \label{sin_r_odd}
\end{align}
as $r$ is odd.
In view of \eqref{sin_k_even_1} and \eqref{sin_r_odd},   the vertical integral  \eqref{common expression for left vertical integral} reduces to
\begin{align}\label{V(k,r)_k even_r odd}
V_{k,r}(x)= \frac{(-1)^{\frac{k+r+1}{2}-1}(2\pi)^{k+1-r}}{ \pi^{2} x} \frac{1}{2\pi i}    \int_{(1+\lambda)} & \Gamma(1-s)\Gamma(1-k+ks)\Gamma(s+r)\zeta(1-k+ks)\nonumber\\
\times & \zeta(s+r)  \sin\left(\frac{\pi k s}{2}\right)\sin \left(\frac{\pi s}{2}\right)\left(\frac{(2\pi)^{k+1}}{x}\right)^{-s} \textrm{d}s.
\end{align}
 Here we remark that the only changes happened in the power of $-1$ and $\cos \left(\frac{\pi s}{2}\right)$ got replaced by $\sin \left(\frac{\pi s}{2}\right)$.  At this point,  to simplify further, we multiply and divide by $\cos \left(\frac{\pi s}{2}\right)$ in \eqref{V(k,r)_k even_r odd}.  Then using Euler's reflection formula and simplifying, we arrive at
 \begin{align*}
 V_{k,r}(x)= \frac{(-1)^{\frac{k+r+1}{2}-1}(2\pi)^{k-r}}{  x} \frac{1}{2\pi i}    \int_{(1+\lambda)} & \frac{\Gamma(1-k+ks)\Gamma(s+r)}{ \Gamma(s)}\zeta(1-k+ks)\nonumber\\
\times & \zeta(s+r) \frac{ \sin\left(\frac{\pi k s}{2}\right) }{\cos \left(\frac{\pi s}{2}\right)}\left(\frac{(2\pi)^{k+1}}{x}\right)^{-s} \textrm{d}s.
 \end{align*}
 Now we shall make use of \eqref{use in k even r odd}, that is,
 \begin{align*}
 \frac{ \sin\left(\frac{\pi k s}{2}\right) }{\cos \left(\frac{\pi s}{2}\right)} =(-1)^{\frac{k}{2}}\sum_{j=-(k-1)}^{(k-1)} {}^{''} \, i^j \exp \left( \frac{i \pi js}{2} \right).
 \end{align*}
 This is one of the crucial changes in this proof.  Substituting this expression,  one can show that
  \begin{align*}
 V_{k,r}(x)= \frac{(-1)^{\frac{2k+r+1}{2}-1}(2\pi)^{k-r}}{  x}  \sum_{j=-(k-1)}^{(k-1)} {}^{''} \, i^j  \,\,   \frac{1}{2\pi i}    \int_{(1+\lambda)} & \frac{\Gamma(1-k+ks)\Gamma(s+r)}{ \Gamma(s)}\zeta(1-k+ks)\nonumber\\
\times & \zeta(s+r) \left( \frac{  e^{-\frac{i \pi j}{2}}  (2\pi)^{k+1}}{x}\right)^{-s} \textrm{d}s.
 \end{align*}
 From here onwards,  the analysis of simplifying $V_{k,r}(x)$ is exactly same as in Theorem \ref{k even r even},  so we avoid reproducing the calculations.  If we continue the calculations  along the same spirit of Theorem \ref{k even r even},  then the vertical integral takes the shape of
 \begin{align}\label{final_vertical_k even_r odd}
 V_{k,r}(x)= \frac{(-1)^{\frac{2k+r-1}{2}} (2\pi)^{\frac{k+1-2r}{2}}}{x\,k^{\frac{2k-1}{2}}} 
   \sum_{j=-(k-1)}^{(k-1)} {}^{''} i^j \sum_{n=1}^{\infty} S_{k,r}(n) G_{0,k}^{\,k,0} \!\left(  \,\begin{matrix}\{\}\\r, -\frac{1}{k},\cdots ,-\frac{(k-1)}{k} \end{matrix} \; \Big|X(j)  \right).
 \end{align}
 Now \eqref{final_vertical_k even_r odd} and  along with all the residual terms in Theorem \ref{k even r even},  one can complete the proof of Theorem \ref{k even r odd}.

 \end{proof}
 
 \subsection{A new identity for $\zeta(-1/2)$.}

 First, let us state a lemma which will be useful in proving the next corollary.
\begin{lemma}\label{zeta(-1/2)_lemma}
Let $b$ and $u$ be two real numbers.  Then
\begin{align*}
 \Re\Bigg[ \frac{i \exp( b e^{i u} )}{ \left( \exp(b e^{i u}) -1 \right)^2}  \Bigg] 
= \frac{1}{2} \frac{\sin( b \sin u ) \sinh( b \cos u)}{ \left( \cos( b \sin u) \cosh( b \cos u) -1  \right)^2 + \left(\sin( b \sin u ) \sinh( b \cos u)  \right)^2   }.
\end{align*}
\end{lemma} 
 The proof of this lemma goes along the same vein as the proof of Lemma \ref{our lemma 1},  so we left to readers to verify.

\begin{proof}[Corollary {\rm \ref{Formula for zeta(-1/2)}}][]

It is well-known that $\zeta(-1)=-\frac{1}{12}$ and $\zeta(4)=\frac{\pi^4}{90}$.  Letting $k=2$ and $r=1$, Theorem \ref{k even r odd} yields that
\begin{align}
\sum_{n=1}^{\infty} D_{2,1}(n) e^{-n x} & = \frac{1}{24}+\frac{\sqrt{\pi}}{2 \sqrt{x}}\zeta\left(-\frac{1}{2} \right)+\frac{\pi^{4}}{90x^2} \nonumber \\
& + \frac{\sqrt{\pi}}{2 x} \sum_{n=1}^\infty S_{2,1}(n) \Bigg[ i G_{0,2}^{\,2,0} \!\left(  \,\begin{matrix}\{\}\\1,-\frac{1}{2} \end{matrix} \; \Big| X(1) \right)-  i G_{0,2}^{\,2,0} \!\left(  \,\begin{matrix}\{\}\\1,-\frac{1}{2} \end{matrix} \; \Big| X(-1) \right) \Bigg]. \label{k=2,r=1}
\end{align}
Recall Lemma \ref{MeijerG(2,0,0,2)}, with $b_1=1$ and $b_2=-\frac{1}{2}$,  and use \cite[p.~444,  Equation (10.2.17)]{abra} the fact that 
\begin{align*}
 K_{\frac{3}{2}}(2\sqrt{z})=\frac{e^{-2\sqrt{z}} \sqrt{\pi} \left(1+\frac{1}{2\sqrt{z}}\right)}{2z^{\frac{1}{4}}},
 \end{align*}
to deduce that
\begin{align}\label{Simplified_MeijerG(2,0,0,2)_1}
G_{0,2}^{\,2,0} \!\left(  \,\begin{matrix}\{\}\\1,\frac{-1}{2} \end{matrix} \; \Big| X(j)   \right) \Bigg]= \sqrt{\pi} e^{-2 \sqrt{X(j)} } \left(  1+ \frac{1}{2 \sqrt{X(j)}} \right),
\end{align}
 where 
 \begin{align*}
 2 \sqrt{X(j)}= 2\pi \sqrt{ \frac{2\pi n}{x}} e^{- \frac{i \pi j}{4}}.
 \end{align*}
From \eqref{D(k,r) for r positive},  we see
\begin{align}\label{D(2,1)_series}
\sum_{n=1}^{\infty} D_{2,1}(n) e^{-n x} = \sum_{n=1}^{\infty}\frac{1}{n^{2}}\frac{\mathrm{d}}{\mathrm{d}x}\left(\frac{1}{1-e^{n^{2}x}}\right).
\end{align}
Simplifying the definition \eqref{definition of D_{k,r}} of $S_{2,1}(n)$, one can write
\begin{align}\label{S(2,1)}
S_{2,1}(n)  = \sum_{n=d^2 m} \frac{d}{m}.
\end{align}
Now plugging \eqref{Simplified_MeijerG(2,0,0,2)_1}, \eqref{D(2,1)_series},  and \eqref{S(2,1)} in \eqref{k=2,r=1},  we derive that
\begin{align}\label{formula_zeta(-1/2)}
\sum_{n=1}^{\infty}\frac{1}{n^{2}}\frac{\mathrm{d}}{\mathrm{d}x}\left(\frac{1}{1-e^{n^{2}x}}\right)
& =  \frac{1}{24}+\frac{\sqrt{\pi}}{2 \sqrt{x}}\zeta\left(-\frac{1}{2} \right)+\frac{\pi^{4}}{90x^2} \nonumber \\
&+ \frac{\pi}{ x} \sum_{m=1}^\infty \frac{1}{m}    \sum_{d=1}^{\infty} \Re \Bigg[ i d e^{-\theta d \sqrt{m}} \left( 1+  \frac{1}{\theta d \sqrt{m}}   \right) \Bigg],
\end{align}
where $\theta= 2\pi \sqrt{ \frac{2\pi }{x}} e^{- \frac{i \pi }{4}}.$
Next  aim is to decipher the  sum:
\begin{align*}
& \Re\Bigg[ i \sum_{d=1}^\infty d e^{-\theta d \sqrt{m}} + \frac{i}{\theta \sqrt{m}} e^{-\theta d \sqrt{m}} \Bigg] \\
& = \Re\Bigg[ \frac{ i\, e^{\theta  \sqrt{m}}}{ \left( e^{\theta  \sqrt{m}}-1 \right)^2} +  \frac{i}{ \theta \sqrt{m} \left( e^{\theta  \sqrt{m}}-1 \right)}  \Bigg].
\end{align*}
Let us write 
\begin{align*}
A(\theta)= \frac{ i\, e^{\theta  \sqrt{m}}}{ \left( e^{\theta  \sqrt{m}}-1 \right)^2},  \quad {\rm and}\quad  B(\theta) = \frac{i}{ \theta \sqrt{m} \left( e^{\theta  \sqrt{m}}-1 \right)}.
\end{align*}
To extract real part of $A(\theta)$, we need to use Lemma \ref{zeta(-1/2)_lemma}.   
We write $\theta \sqrt{m}= b e^{- \frac{i \pi}{4}}$,  where $b= 2\pi \sqrt{ \frac{2\pi m}{x}} $.
Then,  employing Lemma \ref{zeta(-1/2)_lemma},  with $u=- \frac{ \pi}{4}$,  we get
\begin{align}\label{real part_A(theta)}
\Re[A(\theta)] &  = \Re\Bigg[ \frac{ i\,  e^{b e^{- \frac{i \pi}{4}}}}{ \left( e^{ b e^{- \frac{i \pi}{4}}}-1 \right)^2} \Bigg]   \nonumber \\
& = - \frac{1}{2} \frac{\sin\left(  \frac{b}{\sqrt{2} } \right) \sinh\left(  \frac{b}{\sqrt{2} } \right) }{ \left( \cos\left(  \frac{b}{\sqrt{2} } \right)  \cosh\left(  \frac{b}{\sqrt{2} } \right)  -1  \right)^2 + \left(\sin\left(  \frac{b}{\sqrt{2} } \right)  \sinh\left(  \frac{b}{\sqrt{2} } \right)  \right)^2   }.
\end{align}
Again,  to find the real part of $B(\theta)$,  we shall use Lemma \ref{our lemma 1},  with $u= \frac{\pi}{4}$ and $v=1$.  Thus, we have
\begin{align}\label{real part_B(theta)}
\Re[ B(\theta) ] 
& = \frac{1}{b}  \Re \Bigg[ \frac{i e^{\frac{i \pi}{4}}}{  e^{ b e^{-\frac{i \pi}{4}}}-1 } \Bigg] \nonumber \\
& =  -\frac{1}{4} \frac{\sin \left(  \frac{b}{\sqrt{2} }\right) + \cos \left(  \frac{b}{\sqrt{2} }\right)-e^{- \frac{b}{\sqrt{2} }}}{ \frac{b}{\sqrt{2}} \left( \cosh\left(  \frac{b}{\sqrt{2} }\right)-\cos\left(  \frac{b}{\sqrt{2} } \right) \right)}.
\end{align}
Now in view of \eqref{real part_A(theta)} and \eqref{real part_B(theta)},  the equation \eqref{formula_zeta(-1/2)} becomes
\begin{align*}
\sum_{n=1}^{\infty}\frac{1}{n^{2}}\frac{\mathrm{d}}{\mathrm{d}x}\left(\frac{1}{1-e^{n^{2}x}}\right)
& =  \frac{1}{24}+\frac{\sqrt{\pi}}{2 \sqrt{x}}\zeta\left(-\frac{1}{2} \right)+\frac{\pi^{4}}{90x^2} \nonumber \\
&- \frac{\pi}{ 4  x} \sum_{m=1}^\infty \frac{1}{m}  \Bigg[  \frac{\sin \left(  \frac{b}{\sqrt{2} }\right) + \cos \left(  \frac{b}{\sqrt{2} }\right)-e^{- \frac{b}{\sqrt{2} }}}{  \frac{b}{\sqrt{2}} \left( \cosh\left(  \frac{b}{\sqrt{2} }\right)-\cos\left(  \frac{b}{\sqrt{2} }\right) \right)} \\
& + \frac{ 2 \sin\left(  \frac{b}{\sqrt{2} } \right) \sinh\left(  \frac{b}{\sqrt{2} } \right) }{ \left( \cos\left(  \frac{b}{\sqrt{2} } \right)  \cosh\left(  \frac{b}{\sqrt{2} } \right)  -1  \right)^2 + \left(\sin\left(  \frac{b}{\sqrt{2} } \right)  \sinh\left(  \frac{b}{\sqrt{2} } \right)  \right)^2   } \Bigg]
\end{align*}
Finally,  replace $x$ by $\alpha$ and $\beta = \frac{4 \pi^3}{\alpha}$,  then one check that $\frac{b}{\sqrt{2}}= \sqrt{m \beta}$.  Plugging all these variables, we complete the proof of Corollary \eqref{Formula for zeta(-1/2)}. 

\end{proof}

\begin{proof}[Theorem {\rm {\ref{special case k even r=-1 }}}][]
In this theorem, we are concerned with $k\geq 2$ even and $r=-1$,  so the integrand function changes to 
\begin{align*}
F(s)=\Gamma(s) \zeta(ks) \zeta(s+1) x^{-s}. 
\end{align*}
One can easily show that the only poles of this integrand function are at $s=\frac{1}{k}$ and $s=0$.  We shall note that $s=0$ is a pole of order $2$,  where as $\frac{1}{k}$ remains a simple pole. Thus,  the residue at $\frac{1}{k}$ will be 
\begin{align*}
R_{1/k}= \frac{1}{k} \Gamma\left(\frac{1}{k}\right)\zeta\left(\frac{1}{k} +1\right)x^{-\frac{1}{k}}. 
\end{align*}
And the residue at $s=0$ will be 
\begin{align*}
R_{0}= \mathop{\rm Res}_{s=0} \Gamma(s) \zeta(ks) \zeta(s+1)x^{-s} & =\lim_{s \rightarrow 0} \frac{d}{\mathrm{d}s} \left[ s^2 \Gamma(s)\zeta(ks)\zeta(s+1)x^{-s} \right]. 
\end{align*}
From \eqref{Series expansion of Zeta} and \eqref{series expansion of Gamma around zero},  we write the Laurent series expansion at $s=0$ of each factor of the integrand function:
\begin{align*}
\Gamma(s) & = \frac{1}{s}-\gamma+\frac{1}{2}\left(\gamma^2 +\frac{\pi^2}{6}\right)s+ \cdots,\\
\zeta(ks) & = -\frac{1}{2} + k \zeta'(0) s + \cdots,  \\
\zeta(s+1) & = \frac{1}{s} + \gamma - \gamma_1 s + \cdots, \\
x^{-s} & = 1 - \log(x) s+ \cdots. 
\end{align*}
Multiplying these expansions,  one can find that the coefficient of $s$ in $s^2 \Gamma(s)\zeta(ks)\zeta(s+1)x^{-s}$ is $k\zeta'(0) +  \frac{\log(x)}{2}$.  Thus, the residue $R_0$ equals to
\begin{align}\label{residue_double pole at 0}
R_{0}= \frac{1}{2} \log \left( \frac{x}{(2\pi)^k} \right),
\end{align}
since $\zeta'(0)=-\frac{1}{2}\log(2\pi)$. 
The remaining part of the proof is exactly same as in Theorem \ref{k even r odd},  as we are dealing a particular case of when $k$ is even and $r$ is odd.  Hence,  considering the above residual terms, Theorem \ref{k even r odd}, with $r=-1$,  yields that
\begin{align}
 \sum_{n=1}^{\infty} D_{k,-1}(n) e^{-nx}&=\frac{1}{2}\log\left(\frac{x}{(2\pi)^k}\right)+\frac{1}{k} \Gamma\left(\frac{1}{k}\right)\zeta\left(\frac{1}{k} +1\right)x^{-\frac{1}{k}} +  \frac{(-1)^{k-1} (2\pi)^{\frac{k+3}{2}}}{ x\, k^{\frac{2k-1}{2}}} \nonumber \\
 & \times \sum_{j=-(k-1)}^{(k-1)} {}^{''} i^j \sum_{n=1}^{\infty} S_{k,-1}(n)  G_{0,k}^{\,k,0} \!\left(  \,\begin{matrix}\{\}\\-1,-\frac{1}{k},\cdots ,-\frac{(k-1)}{k} \end{matrix} \; \Big|X(j)  \right). \label{r=-1}
\end{align}
Here we invoke Lemma \ref{MeijerG(k,0,0,k)} with $b=-1$ to simplify that
\begin{align}\label{MeiG(k,0,0,k)}
G_{0,k}^{\,k,0} \!\left(  \,\begin{matrix}\{\}\\-1,-\frac{1}{k},\cdots ,-\frac{(k-1)}{k} \end{matrix} \; \Big|X(j)  \right) = \frac{(2\pi)^{\frac{k-1}{2}}}{\sqrt{k}} \frac{ \exp\left( -k X(j)^{\frac{1}{k}} \right)}{X(j)}.
\end{align}
Substituting \eqref{MeiG(k,0,0,k)} in \eqref{r=-1},  the right side infinite sum reduces to 
\begin{align}
&  \frac{(-1)^{k-1} (2\pi)^{k+1}}{ x\, k^k}
  \sum_{j=-(k-1)}^{(k-1)} {}^{''} i^j  \sum_{n=1}^{\infty} S_{k,-1}(n)  \frac{ \exp\left( -k X(j)^{\frac{1}{k}} \right)}{X(j)} \nonumber \\
  & = (-1)^{k-1}  \sum_{j=-(k-1)}^{(k-1)} {}^{''} i^j e^{ \frac{i \pi j}{2} } \sum_{n=1}^{\infty} \frac{ S_{k,-1}(n)}{n}   \exp \left( - e^{-\frac{i \pi j}{2k} } (2\pi)^{1+ \frac{1}{k} } \left( \frac{n}{x} \right)^{\frac{1}{k}}  \right) \nonumber \\
  & = (-1)^{k-1}  \sum_{j=-(k-1)}^{(k-1)} {}^{''}  e^{ i \pi j } \sum_{m=1}^{\infty} \sum_{d=1}^{\infty} \frac{1}{d} \exp \left( - e^{-\frac{i \pi j}{2k} } (2\pi)^{1+ \frac{1}{k} } d \left( \frac{m}{x} \right)^{\frac{1}{k}}  \right) \nonumber \\
  & = (-1)^k \sum_{j=-(k-1)}^{(k-1)} {}^{''}  e^{ i \pi j } \sum_{m=1}^{\infty}   \log \left[ 1 - \exp\left(-e^{-\frac{i \pi j}{2k} } (2\pi)^{1+ \frac{1}{k} } \left( \frac{m}{x} \right)^{\frac{1}{k}} \right) \right].  \label{last sum_r=-1}
\end{align}
In the penultimate step,  we have used the definition of $S_{k,-1}(n)$,  and in the final step we used the identity $-\log(1-y)= \sum_{m=1}^\infty \frac{y^m}{m}$ when $|y|<1$.  Finally,  in view \eqref{r=-1} and \eqref{last sum_r=-1},  we complete the proof of Theorem \ref{special case k even r=-1 }.  
\end{proof}

\begin{proof}[Theorem {\rm \ref{k odd r odd}}][]
Note that here we are dealing with $k\geq 1$ and $r\neq -1$ are both as an odd integer. 
Let us recall that the integrand function is 
\begin{align*}
F(s)= \Gamma(s) \zeta(ks) \zeta(s-r)x^{-s}.
\end{align*}
We shall divide the analysis of poles in two cases. 

{\bf Case I}: First,  we consider $r$ as an odd positive integer.  We know that $\Gamma(s)$ has simple poles at $0,-1,-2,-3,  \cdots$.  The poles at even negative integers are getting cancelled by the trivial zeros of $\zeta(ks)$,  whereas the poles at odd negative integers will be cancelled by $\zeta(s-r)$. 
Therefore,  in this case,  the only poles of the integrand function are at $s=0,  \frac{1}{k}$,  and $1+r$.  These are all simple poles and their corresponding residues have been already calculated in Theorem \ref{k even r even},  so we won't repeat it here. 

{\bf Case II}: We consider $r\neq -1$ as an odd negative integer.  One can easily check that $0$ and $\frac{1}{k}$ are simple poles in this case too. 
As $r \leq -3$ is odd,  so $1+r$ is  an even negative integer. Thus, $1+r$ is a pole of order $2$ of the factor $\Gamma(s) \zeta(s-r)$.  And we also note that $1+r$ is  a trivial zero of $\zeta(ks)$.   Therefore,   $1+r$ remains as a simple pole of the integrand function. 
 
An important point is to note that,  the poles of $\Gamma(s)$ at $s=-1,-3, \cdots,  r$ will not get neutralized by $\zeta(s-r)$, but the poles at the negative odd integers beyond $r$,  say $r-2,  r-4,  \cdots$ will be neutralized by $\zeta(s-r)$ since they are trivial zeros of $\zeta(s-r)$.  Therefore,  in this case, we have to take into account the residual terms coming from the contribution of the poles at $s=-1, -3, \cdots, r$.  

To calculate the residual term corresponding to $s=1+r$,  we have to use the Laurent series expansion of $\Gamma(s)$ i.e.,  Lemma \ref{Laurent_gamma_1+r},  and the Laurent series expansion \eqref{Series expansion of Zeta} of $\zeta(s)$ with $s$ replace by $s-r$.  Thus, we will have
\begin{align}
R_{1+r} & =\lim_{s\rightarrow 1+r} (s-(1+r)) \Gamma(s)  \zeta(ks) \zeta(s-r)x^{-s} \nonumber \\
 & = \lim_{s\rightarrow 1+r} (s-(1+r)) \Gamma(s) \frac{ \zeta(ks)}{(s-(1+r))} (s-(1+r)) \zeta(s-r)x^{-s} \nonumber \\
&  = \frac{(-1)^{1+r}}{(-(1+r))!} k \zeta'(k(1+r)) x^{-(1+r)}.  \label{at 1+r}
\end{align}
Let $R$ be the sum of the residual terms corresponding to the poles at $s=-1, -3, \cdots, r$.  Then,
we have
\begin{align}\label{all residue_negative integers}
R & =\sum_{i=0}^{-\frac{1+r}{2}} \mathop{\rm Res}_{s=-(2i+1)} \Gamma(s) \zeta(ks) \zeta(s-r)x^{-s} \nonumber \\
& = \sum_{i=0}^{-\frac{1+r}{2}}   \frac{-1}{(2i+1)!} \zeta(-k(2i+1)) \zeta(-2i-1-r) x^{2i+1} \nonumber \\
& = \frac{(-1)^{ \frac{1+r}{2} }}{2} \sum_{i=0}^{-\frac{1+r}{2}}  \frac{(-1)^{i+1} B_{k(2i+1)+1}}{(2i+1)! (k(2i+1)+1)  }  \frac{  B_{-2i-1-r}  (2\pi)^{-r  }}{ (-(2i+1+r))!} \left( \frac{x}{2\pi} \right)^{2i+1}.
\end{align}
The remaining part of the proof is in the same direction of the proof of Theorem \ref{k even r even},  so we  briefly mention all the steps. 
Mainly,  we will concentrate on the calculation of the vertical integral $V_{k,r}(x)$.  From \eqref{common expression for left vertical integral},  we have
\begin{align*}
 V_{k,r}(x)= \frac{(2\pi)^{k+1-r}}{ \pi^{2} x} \frac{1}{2\pi i}  & \int_{(1+\lambda)} \Gamma(1-s)\Gamma(1-k+ks)\Gamma(s+r)\zeta(1-k+ks)\zeta(s+r)\nonumber\\
\times &  \sin\left(\frac{\pi k(1-s)}{2}\right)\sin \left(\frac{\pi(1-s-r)}{2}\right)\left(\frac{(2\pi)^{k+1}}{x}\right)^{-s} \textrm{d}s.
\end{align*}
Since $k $ and $r$ both odd,  so we must use 
\begin{align*}
 \sin\left(\frac{\pi k(1-s)}{2}\right) & =(-1)^{\frac{k-1}{2}}\cos\left(\frac{k \pi s}{2}\right), \\
\sin \left(\frac{\pi(1-s-r)}{2}\right) &= (-1)^{\frac{r+1}{2}} \sin\left( \frac{\pi s}{2} \right).
\end{align*}
Use these two trigonometric identities to see
\begin{align*}
V_{k,r}(x)= \frac{ (-1)^{\frac{k+r}{2}}(2\pi)^{k+1-r}}{ \pi^{2} x} \frac{1}{2\pi i}   \int_{(1+\lambda)} & \Gamma(1-s)\Gamma(1-k+ks)\Gamma(s+r)\zeta(1-k+ks)\zeta(s+r)\nonumber\\
\times &  \cos\left(\frac{k \pi s}{2}\right)  \sin\left( \frac{\pi s}{2} \right) \left(\frac{(2\pi)^{k+1}}{x}\right)^{-s} \textrm{d}s.
\end{align*}
Here we multiply and divide by $ \cos\left( \frac{\pi s}{2} \right)$ and then use Euler's reflection formula to deduce that
\begin{align*}
V_{k,r}(x)= \frac{ (-1)^{\frac{k+r}{2}}(2\pi)^{k-r}}{ x} \frac{1}{2\pi i}   \int_{(1+\lambda)} & \frac{\Gamma(1-k+ks)\Gamma(s+r)}{\Gamma(s)} \zeta(1-k+ks)\zeta(s+r)\nonumber\\
\times &  \frac{ \cos\left(\frac{k \pi s}{2}\right) }{  \cos\left( \frac{\pi s}{2} \right) } \left(\frac{(2\pi)^{k+1}}{x}\right)^{-s} \textrm{d}s.
\end{align*}
Next,  we shall use \eqref{use in k odd r odd},  that is,
\begin{align*}
 \frac{ \cos\left(\frac{k \pi s}{2}\right) }{  \cos\left( \frac{\pi s}{2} \right) } =(-1)^{\frac{k-1}{2}}\sum_{j=-(k-1)}^{(k-1)} {}^{''} \, i^j \exp\left(- \frac{i\pi j s}{2} \right). 
\end{align*}
This is one of the important changes in this proof.  Plugging this expression in the above integral,  we obtain
\begin{align*}
V_{k,r}(x)= \frac{ (-1)^{\frac{2k+r-1}{2}}(2\pi)^{k-r}}{ x}  \sum_{j=-(k-1)}^{(k-1)} {}^{''} \, i^j  \,\,   \frac{1}{2\pi i}    \int_{(1+\lambda)} & \frac{\Gamma(1-k+ks)\Gamma(s+r)}{ \Gamma(s)}\zeta(1-k+ks)\nonumber\\
\times & \zeta(s+r) \left( \frac{  e^{\frac{i \pi j}{2}}  (2\pi)^{k+1}}{x}\right)^{-s} \textrm{d}s.
\end{align*}
Henceforth the simplification of the vertical integral $V_{k,r}(x)$ is exactly same as in Theorem \ref{k even r even}.  If we continue the calculations in the same spirit,  then the final expression of $V_{k,r}(x)$ will be
\begin{align}\label{final_vertical_k odd}
V_{k,r}(x)= \frac{ (-1)^{\frac{2k+r-1}{2}}(2\pi)^{\frac{k+1-2r}{2} }}{ x \,  k^{\frac{2k-1}{2}} }    \sum_{j=-(k-1)}^{(k-1)} {}^{''} \, i^j  \,\, \sum_{n=1}^{\infty} S_{k,r}(n)\; G_{0,k}^{\,k,0} \!\left(  \,\begin{matrix}\{\}\\r,-\frac{1}{k},\cdots ,-\frac{(k-1)}{k} \end{matrix} \; \Big| X(-j)   \right).
\end{align}
An important point is to note that the argument of the Meijer $G$-function is $X(-j)$. 
Finally,  collecting all the residual terms \eqref{at 1+r},  \eqref{all residue_negative integers}, and together with \eqref{final_vertical_k odd},  one can complete the proof of \eqref{Rama_generalization}.

\end{proof}
 
\begin{proof}[Corollary {\rm \ref{k odd r odd}}][]
Setting $r=-(2m+1)$,  with $m\geq 1$,  in Theorem \ref{k odd r odd},  and using the identity
\begin{align*}
\zeta'(-2km) = \frac{(-1)^{km}}{2} \frac{(2km)! \zeta(2km+1)}{(2\pi)^{2km}},
\end{align*}
we complete the proof.  
\end{proof}

\subsection{Derivation of Ramanujan's formula for odd zeta values}
\begin{proof}[Corollary {\rm {\ref{Ramanujan_odd zeta}}}][]
Plugging $k=1$ in Corollary \ref{k odd r -ve odd},  we deduce that
\begin{align}\label{Rama_k=1}
\sum_{n=1}^{\infty} D_{1,-(2m+1)}(n) e^{-nx} & =- \frac{1}{2}\zeta(2m+1) + \frac{\zeta\left(2m+2\right)}{x} 
 +  \frac{ (-1)^m}{2}  \zeta(2m+1)  \left( \frac{x}{2\pi} \right)^{2m} \nonumber \\ 
& + \frac{(-1)^m (2\pi)^{2m+1}}{2} \sum_{i=0}^m \frac{(-1)^{i+1} B_{2i+2} B_{2m-2i} }{(2i+2)! (2m-2i)!} \left( \frac{x}{2\pi} \right)^{2i+1} \nonumber \\
& + \frac{(-1)^m (2\pi)^{2m+2 }}{x } \sum_{n=1}^{\infty} S_{1,-(2m+1)}(n) G_{0,1}^{\,1,0} \!\left(  \,\begin{matrix}\{\}\\ -(2m+1) \end{matrix} \; \Big| \frac{(2\pi)^2 n}{x}\right).
\end{align}
From \eqref{relation_D_sigma},  we know 
\begin{align}\label{D_S}
D_{1,-(2m+1)}(n)= \sigma_{-(2m+1)}(n),  \quad {\rm and} \quad S_{1,-(2m+1)}(n)= \sigma_{2m+1}(n). 
\end{align}
The equation  \eqref{MeijerG(1,0,0,1)} yields that 
\begin{align}\label{special_MeijerG}
G_{0,1}^{\,1,0} \!\left(  \,\begin{matrix}\{\}\\ -(2m+1) \end{matrix} \; \Big| \frac{(2\pi)^2 n}{x}\right)= \left( \frac{x}{4 \pi^2 n} \right)^{2m+1 } e^{-\frac{4 \pi^2 n}{x} }. 
\end{align}
Now substituting \eqref{D_S} and \eqref{special_MeijerG} in \eqref{Rama_k=1},  one can derive that
\begin{align*}
 \sum_{n=1}^{\infty} \sigma_{-(2m+1)}(n) e^{-nx} + \frac{1}{2}\zeta(2m+1) & = \frac{\zeta\left(2m+2\right)}{x} 
 +  \frac{ (-1)^m}{2}  \zeta(2m+1)  \left( \frac{x}{2\pi} \right)^{2m} \nonumber \\ 
& + \frac{(-1)^m (2\pi)^{2m+1}}{2} \sum_{i=0}^m \frac{(-1)^{i+1} B_{2i+2} B_{2m-2i} }{(2i+2)! (2m-2i)!} \left( \frac{x}{2\pi} \right)^{2i+1} \nonumber \\
& + (-1)^m  \left( \frac{x}{2 \pi} \right)^{2m} \sum_{n=1}^{\infty} \frac{ \sigma_{2m+1}(n) }{n^{2m+1}} e^{-\frac{4 \pi^2 n}{x}}.
\end{align*}
Finally,  invoking Euler's formula \eqref{Euler_zeta(2m)} for $\zeta(2m+2)$ and using the fact that
$ \frac{ \sigma_{2m+1}(n)} {n^{2m+1}} =  \sigma_{-(2m+1)}(n) $
 and upon simplification,  one can derive  \eqref{Rama_odd zeta}. 
\end{proof}

\begin{proof}[Corollary {\rm \ref{special cases when m less than -1}}][]
Letting $k=1$ and $r=2m+1$ with $m\geq 1$ in Theorem \ref{k odd r odd},  and with the fact that 
$\zeta(-(2m+2))=0$ and $\zeta(-(2m+1))=- \frac{B_{2m+2}}{2m+2}$, 
we find that
\begin{align}
\sum_{n=1}^{\infty} D_{1,2m+1}(n) e^{-nx} & = \frac{B_{2m+2}}{4m+4} + (2m+1)! \frac{ \zeta(2m+2)}{x^{2m+2}} \nonumber \\
& + (-1)^{m+1} \frac{1}{x (2\pi)^{2m} } \sum_{n=1}^\infty S_{1, 2m+1}(n) \left( \frac{x}{4 \pi^2 n} \right)^{-(2m+1) } e^{-\frac{4 \pi^2 n}{x} } \nonumber  \\
 \Rightarrow \sum_{n=1}^{\infty}  \sigma_{2m+1}(n) e^{-nx} & = \frac{B_{2m+2}}{4m+4} + (-1)^m \left(\frac{2\pi}{x} \right)^{2m+2} \frac{B_{2m+2}}{4m+4} \nonumber \\
 & + (-1)^{m+1} \left(\frac{2\pi}{x} \right)^{2m+2} \sum_{n=1}^{\infty}  \sigma_{2m+1}(n) e^{- \frac{4 \pi^2n}{x}}.  \label{k=1, r=2m+1}
\end{align}
Here we have used \eqref{relation_D_sigma} to write $D_{1,2m+1}(n)$ and $S_{1,2m+1}(n) n^{2m+1}$ in terms of $\sigma_{2m+1}(n)$.  Finally,  to obtain \eqref{Rama_k=1_ r=2m+1}, replace $x=2 \alpha$ and $\alpha \beta =\pi^2$ in \eqref{k=1, r=2m+1}.

\end{proof}

Substituting $k=r=1$ in Theorem \ref{k odd r odd} and setting $x=2\alpha$ and $\alpha \beta =\pi^2$,   one can obtain the following identity of Schl\"{o}milch \cite{schlomilch}, and rediscovered by Ramanujan \cite[Ch. 14, Sec. 8, Cor. (i)]{ramnote}, \cite[p.~318]{lnb}.
\begin{corollary}\label{special case k=r=1}
 For $\alpha,\beta>0$ with $\alpha \beta=\pi^2$, 
\begin{align*}
\alpha \sum_{n=1}^{\infty} \frac{n}{e^{2n\alpha}-1}+\beta \sum_{n=1}^{\infty} \frac{n}{e^{2n\beta}-1}=\frac{\alpha+\beta}{24}-\frac{1}{4\alpha}
\end{align*}
In particular, if we consider  $\alpha=\beta=\pi$,  then
$$\sum_{n=1}^{\infty} \frac{n}{e^{2n\pi}-1}=\frac{1}{24}-\frac{1}{8\pi}.$$
\end{corollary}

\begin{proof}[Theorem {\rm {\ref{k odd r=-1}}}][]
The proof of this theorem travels in the same direction as of Theorem \ref{k odd r odd} since we are dealing $k$ odd and $r=-1$.  In this case,  the integrand function is 
\begin{align*}
F(s)= \Gamma(s) \zeta(ks) \zeta(s+1) x^{-s}.
\end{align*}
One can easily check that the integrand function has poles are at $s=0, \frac{1}{k}$ and $-1$.  At $s=0$, we have a pole order $2$,  but the other two points are simple pole.  In view of \eqref{residue_double pole at 0},  we know 
\begin{align*}
R_{0}= \frac{1}{2} \log \left( \frac{x}{(2\pi)^k} \right).
\end{align*}
One can find that 
\begin{align*}
R_{-1}= \frac{\zeta(-k)x}{2}. 
\end{align*}
Now along with the residue at $\frac{1}{k}$,  Theorem \ref{k odd r odd}, with $r=-1$, leads to 
\begin{align}
 \sum_{n=1}^{\infty} D_{k,-1}(n) e^{-nx}&=\frac{1}{2}\log\left(\frac{x}{(2\pi)^k}\right)+\frac{1}{k} \Gamma\left(\frac{1}{k}\right)\zeta\left(\frac{1}{k} +1\right)x^{-\frac{1}{k}} +  \frac{\zeta(-k)}{2} + \frac{ (2\pi)^{\frac{k+3}{2}}}{ x\, k^{\frac{2k-1}{2}}} \nonumber \\
 & \times \sum_{j=-(k-1)}^{(k-1)} {}^{''} i^j \sum_{n=1}^{\infty} S_{k,-1}(n) G_{0,k}^{\,k,0} \!\left(  \,\begin{matrix}\{\}\\-1,-\frac{1}{k},\cdots ,-\frac{(k-1)}{k} \end{matrix} \; \Big|X(-j)  \right). \label{k odd, r=-1}
\end{align}
Calculations of the above right hand side sum goes along the same line of \eqref{r=-1}.  For completeness,  we mention it briefly.  First,  employ Lemma \ref{MeijerG(k,0,0,k)} with $b=-1$ to see
\begin{align*}
G_{0,k}^{\,k,0} \!\left(  \,\begin{matrix}\{\}\\-1,-\frac{1}{k},\cdots ,-\frac{(k-1)}{k} \end{matrix} \; \Big|X(-j)  \right) = \frac{(2\pi)^{\frac{k-1}{2}}}{\sqrt{k}} \frac{ \exp\left( -k X(-j)^{\frac{1}{k}} \right)}{X(-j)}.
\end{align*}
Substituting it in \eqref{k odd, r=-1}, together with definition $\eqref{X(j)}$ of $X(j)$,  the right side infinite sum  takes the shape of  
\begin{align}
  &   \sum_{j=-(k-1)}^{(k-1)} {}^{''} i^j e^{ -\frac{i \pi j}{2} } \sum_{n=1}^{\infty} \frac{ S_{k,-1}(n)}{n}   \exp \left( - e^{\frac{i \pi j}{2k} } (2\pi)^{1+ \frac{1}{k} } \left( \frac{n}{x} \right)^{\frac{1}{k}}  \right) \nonumber \\
  & =  \sum_{j=-(k-1)}^{(k-1)} {}^{''}   \sum_{m=1}^{\infty} \sum_{d=1}^{\infty} \frac{1}{d} \exp \left( - e^{\frac{i \pi j}{2k} } (2\pi)^{1+ \frac{1}{k} } d \left( \frac{m}{x} \right)^{\frac{1}{k}}  \right) \nonumber \\
  & = - \sum_{j=-(k-1)}^{(k-1)} {}^{''}  \sum_{m=1}^{\infty}   \log \left[ 1 - \exp\left(-e^{\frac{i \pi j}{2k} } (2\pi)^{1+ \frac{1}{k} } \left( \frac{m}{x} \right)^{\frac{1}{k}} \right) \right].  \label{last sum_k odd r=-1}
\end{align}
Now substituting \eqref{last sum_k odd r=-1} in \eqref{k odd, r=-1},  one can finish the proof of \eqref{gen_dekekind}.
\end{proof}

\begin{proof}[Corollary {\rm \ref{formula for logarithm Dedekind eta function}}][]
Plugging $k=1$ in \eqref{gen_dekekind},  we arrive at
\begin{align*}
& \sum_{n=1}^{\infty} D_{1,-1}(n) e^{-nx}= \frac{1}{2} \log\left( \frac{x}{2\pi} \right)+ \frac{\pi^2}{6 x} - \frac{x}{24} 
 -  \sum_{m=1}^\infty \log \left[1- \exp\left(-  \frac{4 \pi^2 m}{x} \right)  \right] \nonumber \\
& \Rightarrow   \sum_{n=1}^{\infty} \sum_{d|n} \frac{1}{d} e^{-nx} = \frac{1}{2} \log\left( \frac{x}{2\pi} \right)+ \frac{\pi^2}{6 x} - \frac{x}{24} +  \sum_{m=1}^\infty  \sum_{d=1}^\infty \frac{1}{d} e^{ -\frac{4 \pi^2 m d}{x}}  \nonumber\\
& \Rightarrow \sum_{n=1}^\infty  \frac{1}{n(e^{nx}-1)} =  \frac{1}{2} \log\left( \frac{x}{2\pi} \right)+ \frac{\pi^2}{6 x} - \frac{x}{24} +  \sum_{n=1}^\infty  \frac{1}{n \left( e^{ \frac{4 \pi^2 n}{x}}-1 \right)}.
 \end{align*}
Finally,  replacing $x= 2\alpha$ and $\beta= \pi^2/\alpha$ and simplifying we  obtain \eqref{Dedekind}. 
\end{proof}

\section{Concluding Remarks}


On page 332,  Ramanujan recorded an intriguing formula for $\zeta(1/2)$ and also mentioned the following interesting  infinite series
\begin{align*}
\sum_{n=1} \frac{n^r}{\exp( n^N x) -1},
\end{align*}
where $N\in \mathbb{N} , r\in \mathbb{Z}$ and $N-r$ is any even integer.  This infinite series can be written as 
\begin{align*}
 \frac{1}{2\pi i}\int_{c-i\infty}^{c+i\infty} \Gamma(s)\zeta(s)\zeta(Ns-r)x^{-s}\textrm{d}s,
\end{align*}
for $c >\max\left( 1, \frac{1+r}{N}\right)$.  Recently,  Dixit and the second author obtained a beautiful generalization of Ramanujan's formula for odd zeta values while studying the above integral. 

In the current paper,  we have studied a variant of the above integral,  
\begin{align*}
  \sum_{n=1}^{\infty} D_{k,r}(n) e^{-nx}  = \frac{1}{2\pi i}\int_{c-i\infty}^{c+i\infty} \Gamma(s)\zeta(ks)\zeta(s-r)x^{-s}\textrm{d}s,
\end{align*}
where for $k\in \mathbb{N},  r \in \mathbb{Z}$, and $c > \max\left(  1/k, 1+r \right)$.  The analysis of this integral can be divided into the following four cases. 

{\bf Case 1:} When $k$ and $r$ are both are even integers, Theorem \ref{k even r even} gives us a transformation formula,  which allows us to derive Ramanujan's formula for $\zeta(1/2)$ and Wigert's formula for $\zeta(1/k)$.  

{\bf Case 2:} When $k$ is even and $r$ is odd,  i.e., $k-r$ is an odd integer. In this case, Theorem \ref{k even r odd} allows us to obtain a new identity for $\zeta(-1/2)$ analogous to Ramanujan's formula for $\zeta(1/2)$. 

{\bf Case 3:} When $k$ and $r$  are both odd,  i.e., $k-r$ is an even integer.  This is the most interesting case.  Theorem \ref{k odd r odd} allows us to derive Ramanujan's formula for odd zeta values. Corollary \ref{k odd r -ve odd} can be conceived as a one-variable generalization of Ramanujan's formula for odd zeta values,  which connects two odd zeta values $\zeta(2m+1)$, and $\zeta(2km+1)$ and a zeta value at a rational argument $\zeta\left( \frac{1}{k}+ 2m+1 \right)$.

{\bf Case 4:} When $k$ is  odd and $r$ is even,  i.e., $k-r$ is an odd integer.  In this case,  the integrand function will have infinitely many poles, namely,  at all odd negative integers.  In the current paper,  we have not discussed this case.  It would be interesting to study this case too.

Ramanujan's formula \eqref{Ramanujan's formula} for $\zeta(2m+1)$ has been generalized by many mathematicians in various directions.  Ramanujan himself gave a huge generalization \cite[Entry 20, p.~430]{bcbIV}.   Berndt \cite{berndt77} derived Euler's formula  for $\zeta(2m)$ as well as Ramanujan's formula for $\zeta(2m+1)$ from a single transformation formula for a generalized Eisenstein series.  The Dirichlet $L$-function $L(s,\chi)$ and the Hurwitz zeta function $\zeta(s,a)$ are natural generalizations of the Riemann zeta function.  A character analogue of Ramanujan's formula was obtained by Merrill \cite{merrill} and more generally, for any periodic function was obtained by Bradley \cite{bradley}.  Dixit,  Kumar,  Gupta and the second author \cite{DGKM20} found a generalization for Hurwitz zeta function,  which allowed them to connect many odd zeta values from a single formula.  Dixit and Gupta \cite{DG19} established an interesting analogue of Ramanujan's formula which connects square of odd zeta values.  Very recently,  Banerjee and Kumar \cite{BK21} found an analogue of Ramanujan's formula  \eqref{Ramanujan's formula}  over an  imaginary quadratic fields.  
For more references on the generalizations of Ramanujan's formula,  readers are encouraged to see \cite{berndtstraubzeta}.

\textbf{Acknowledgements.}
The authors want to thank IIT Indore for the conductive research environment. 
The second author is grateful to SERB for the Start-Up Research Grant SRG/2020/000144.

\end{document}